\documentclass[11pt]{amsart}

\usepackage{amssymb}
\usepackage{url}
\usepackage{adjustbox}

\usepackage{wasysym}
\usepackage{amsmath, calc, mathdots, physics}

\usepackage{pgfplots}
\pgfplotsset{width=10cm,compat=1.9}

\usepackage{mathrsfs,multicol}
\usepackage[pagebackref, 
hypertexnames=false, colorlinks, citecolor=red, linkcolor=red]{hyperref}
\usetikzlibrary{decorations.pathmorphing}

	\normalbaselines						  %
\def\RR{\mathbb{R}}

\def\CC{\mathbb{C}}

\newcommand{\al}{{\alpha}}
\newcommand{\la}{{\lambda}}
\newcommand{\f}{{\varphi}}

\newcommand{\cX}{{\mathcal{X}}}

\newcommand{\R}{{\mathbb  R}}

\newcommand{\te}{{\theta}}

\newcommand{\N}{{\mathbb  N}}
\newcommand{\C}{{\mathbb  C}}

\newcommand{\fdot}{\,\cdot\,}

\makeatletter
\def\Ddots{\mathinner{\mkern1mu\raise\p@
\vbox{\kern7\p@\hbox{.}}\mkern2mu
\raise4\p@\hbox{.}\mkern2mu\raise7\p@\hbox{.}\mkern1mu}}
\makeatother

\newcommand{\cH}{\mathcal{H}}

\newcommand{\cB}{\mathcal{B}}

\newcommand{\cA}{\mathcal{A}}
\newcommand{\eps}{\varepsilon}
\newcommand{\ta}{\theta}

\newcommand{\wt}{\widetilde}
\newcommand{\vta}{\vartheta}
\newcommand{\civta}{(\vartheta)}
\newcommand{\bmu}{{\boldsymbol{\mu}}}

\newcommand{\cD}{\mathcal{D}}
\newcommand{\cG}{\mathcal{G}}

\newcommand{\cM}{\mathcal{M}}

\newcommand{\fh}{\mathfrak{H}}
\newcommand{\fK}{\mathfrak{K}}

\newcommand{\cR}{\mathcal{R}}

\newcommand{\bB}{\mathbf{B}}

\DeclareMathOperator{\Ker}{Ker}
\DeclareMathOperator{\Mul}{mul}
\DeclareMathOperator{\dom}{dom}

\DeclareMathOperator{\Ran}{ran}
\DeclareMathOperator{\spa}{span}

\DeclareMathOperator{\clos}{clos}

\newcommand{\ci}[1]{_{ {}_{\scriptstyle #1}}}
\newcommand{\ti}[1]{_{\scriptstyle \text{\rm #1}}}

\catcode`@=11
\chardef\mathlig@atcode\count255

\def\actively#1#2{\begingroup\uccode`\~=`#2\relax\uppercase{\endgroup#1~}}
\def\mathlig@gobble{\afterassignment\mathlig@next@cmd\let\mathlig@next= }
\def\mathlig@delim{\mathlig@delim}
\def\mathlig@defcs#1{\expandafter\def\csname#1\endcsname}
\def\mathlig@let@cs#1#2{\expandafter\let\expandafter#1\csname#2\endcsname}
\def\mathlig@appendcs#1#2{\expandafter\edef\csname#1\endcsname{\csname#1\endcsname#2}}

\def\mathlig#1#2{\mathlig@checklig#1\mathlig@end\mathlig@defcs{mathlig@back@#1}{#2}\ignorespaces}

\def\mathlig@checklig#1#2\mathlig@end{%
 \expandafter\ifx\csname mathlig@forw@#1\endcsname\relax
 \expandafter\mathchardef\csname mathlig@back@#1\endcsname=\mathcode`#1%
 \mathcode`#1"8000\actively\def#1{\csname mathlig@look@#1\endcsname}%
 \mathlig@dolig#1\mathlig@delim
\fi
\mathlig@checksuffix#1#2\mathlig@end
}

\def\mathlig@checksuffix#1#2\mathlig@end{%
\ifx\mathlig@delim#2\mathlig@delim\relax\else\mathlig@checksuffix@{#1}#2\mathlig@end\fi
}
\def\mathlig@checksuffix@#1#2#3\mathlig@end{%
\expandafter\ifx\csname mathlig@forw@#1#2\endcsname\relax\mathlig@dosuffix{#1}{#2}\fi
\mathlig@checksuffix{#1#2}#3\mathlig@end
}

\def\mathlig@dosuffix#1#2{%
\mathlig@appendcs{mathlig@toks@#1}{#2}%
\mathlig@dolig{#1}{#2}\mathlig@delim
}

\def\mathlig@dolig#1#2\mathlig@delim{%
 \mathlig@defcs{mathlig@look@#1#2}{%
 \mathlig@let@cs\mathlig@next{mathlig@forw@#1#2}\futurelet\mathlig@next@tok\mathlig@next}%
 \mathlig@defcs{mathlig@forw@#1#2}{%
  \mathlig@let@cs\mathlig@next{mathlig@back@#1#2}%
  \mathlig@let@cs\checker{mathlig@chck@#1#2}%
  \mathlig@let@cs\mathligtoks{mathlig@toks@#1#2}%
  \expandafter\ifx\expandafter\mathlig@delim\mathligtoks\mathlig@delim\relax\else
  \expandafter\checker\mathligtoks\mathlig@delim\fi
  \mathlig@next
 }%
 \mathlig@defcs{mathlig@toks@#1#2}{}%
 \mathlig@defcs{mathlig@chck@#1#2}##1##2\mathlig@delim{%
  \ifx\mathlig@next@tok##1%
   \mathlig@let@cs\mathlig@next@cmd{mathlig@look@#1#2##1}\let\mathlig@next\mathlig@gobble
  \fi
  \ifx\mathlig@delim##2\mathlig@delim\relax\else
   \csname mathlig@chck@#1#2\endcsname##2\mathlig@delim
  \fi
 }%
 \ifx\mathlig@delim#2\mathlig@delim\else
  \mathlig@defcs{mathlig@back@#1#2}{\csname mathlig@back@#1\endcsname #2}%
 \fi
}%

\catcode`@\mathlig@atcode

\mathchardef\ordinarycolon\mathcode`\:
\def\vcentcolon{\mathrel{\mathop\ordinarycolon}}
\mathlig{:=}{\vcentcolon=}
\mathlig{::=}{\vcentcolon\vcentcolon=}

\numberwithin{equation}{section}

\theoremstyle{plain}
\newtheorem{theo}{Theorem}[section]
\newtheorem{cor}[theo]{Corollary}

\newtheorem{prop}[theo]{Proposition}

\theoremstyle{definition}
\newtheorem{defn}[theo]{Definition}

\newtheorem*{theorem*}{Theorem}
\newtheorem*{idea*}{Idea}

\theoremstyle{remark}
\newtheorem{rem}[theo]{Remark}

\newtheorem*{ex*}{Example}
\newtheorem*{exs*}{Examples}
\newtheorem*{rems*}{Remarks}

\title[Spectral Properties of Singular Sturm--Liouville Operators]{Spectral Properties of Singular Sturm--Liouville Operators via Boundary Triples and Perturbation Theory}

\author{Dale~Frymark}
\address{Institut f{\"u}r Angewandte Mathematik, Technische Universit{\"a}t Graz, Steyrergasse 30, 8010 Graz, Austria.}
\email{frymark@math.tugraz.at}

\author{Constanze~Liaw}
\address{Department of Mathematical Sciences, University of Delaware, 501 Ewing Hall, Newark, DE  19716, USA; and 
CASPER, Baylor University, One Bear Place \#97328,      
 Waco, TX  76798, USA.}
\email{liaw@udel.edu}

\thanks{Since August 2020, C.~Liaw has been serving as a Program Director in the Division of Mathematical Sciences at the National Science Foundation (NSF), USA, and as a component of this position, she received support from NSF for research, which included work on this paper. Any opinions, findings, and conclusions or recommendations expressed in this material are those of the authors and do not necessarily reflect the views of the NSF}
\thanks{D.~Frymark was partially supported by the Austrian Science Fund (FWF) Project P33568.}
\thanks{Both authors would like to thank the Erwin Schr\"odinger International Institute for Mathematics and Physics for their hospitality during the workshop on ``Spectral Theory of Differential Operators in Quantum Theory"}


\begin{document}

\begin{abstract}
We apply both the theory of boundary triples and perturbation theory to the setting of semi-bounded Sturm--Liouville operators with two limit-circle endpoints. For general boundary conditions we obtain refined and new results about their eigenvalues and eigenfunctions.

In the boundary triple setup, we obtain simple criteria for identifying which self-adjoint extensions possess double eigenvalues when the parameter is a matrix. We also identify further spectral properties of the Friedrichs extension and (when the operator is positive) the von Neumann--Krein extension. 

Motivated by some recent scalar Aronszajn--Donoghue type results, we find that real numbers can only be eigenvalues for two extensions of Sturm--Liouville operator when the boundary conditions are restricted to corresponding to affine lines in the space from which the perturbation parameter is taken. Furthermore, we determine much of the spectral representation of those Sturm--Liouville operators that can be reached by perturbation theory.
\end{abstract}

\maketitle

\setcounter{tocdepth}{1}
\tableofcontents

\section{Introduction}

Spectral properties of Sturm--Liouville operators have a rich and long history of study, see e.g.~\cite{EGNT,EK,GLN,GZ,GZ2,KL,Rio,Z}. In this manuscript, we consider Sturm--Liouville differential expression with two limit-circle endpoints and with minimal realization that is semi-bounded (see Definition \ref{d-semibdd}). Without loss of generality, we restrict our attention to operators that are (semi-)bounded from below.

The limit-circle endpoints determine that the deficiency indices are $(2,2)$, and hence two boundary conditions are necessary to form self-adjoint extensions. They also constitute more singular versions of regular endpoints, which have an easy setup for boundary value problems but do not cover some popular and important expressions, i.e.~the classical Jacobi differential expression. Niessen and Zettl \cite{NZ} (see also \cite[Theorem 10.6.5]{Z}) showed that Sturm--Liouville expressions with (finite) limit-circle endpoints can be `regularized' via a simple transformation and then their spectral properties can be analyzed as if the endpoints were regular. Two limit-circle endpoints also produce purely discrete spectrum (see e.g.~\cite[Proposition 10.4.3]{Z}), which simplifies some calculations. All of the results herein can be easily applied to regular endpoints --- regular endpoints are in the limit-circle case --- so we choose to use the more general limit-circle endpoints because their theory has more subtleties. 

We consider operators for which the minimal symmetric operator is semi-bounded (with bound $K$), which immediately implies that the limit-circle endpoints are non-oscillating (solutions do not have zeros which accumulate at endpoints). This allows the use of principal and non-principal solutions first developed by Rellich \cite{R} and Kalf \cite{K}. In turn, this fixes the forms of quasi-derivatives and the necessary form of the boundary triple that we will use. Furthermore, it is possible to identify important Friedrichs extension \cite{MZ,NZ} and the von Neumann--Krein extension when the operator is positive. 

Herein, Sturm--Liouville operators that fall under these restrictions are analyzed from two different, but complementary, perspectives: boundary triples and perturbation theory. It should immediately be noted that determining properties for such a wide class of operators is difficult and, to the best knowledge of the authors, the theory has remained in largely the same state since the impactful work of Bailey, Everitt and Zettl \cite{BEZ2} in the mid 90's (see also \cite[Section 10.6]{Z}); some of these results are included in Section \ref{s-preliminaries} for the convenience of the reader. Therein, criteria for eigenvalues of self-adjoint extensions were presented and, in particular, necessary and sufficient conditions for an eigenvalue to have multiplicity two. Our main result, Theorem \ref{t-degen}, contains a condition that ensures simplicity of all eigenvalues of a self-adjoint extension; this condition is independent of the location of the eigenvalue of interest. We then proceed to present a whole class of boundary conditions for which all corresponding operators have simple spectrum in Corollary \ref{c-criteria}, also see Remark \ref{r-improve}.

Boundary triples \cite{BdS, Bruk, Calkin, DM1, DM2, GG, Koch} consist of two maps $\Gamma_0$, $\Gamma_1$, which are defined by quasi-derivatives here, and a boundary space, $\CC^2$ in this context. They have been previously applied to such Sturm--Liouville operators, most notably in \cite[Chapter 6]{BdS}, but also to examples in \cite{AB,BFL,F}. Upon this foundation, Section \ref{s-twolc} builds a core of simple facts around the spectral theory of the two natural self-adjoint extensions ${\bf L}_0$ and ${\bf L}_\infty$ (the Friedrichs extension), whose domains are defined by the kernels $\ker(\Gamma_0)$  and $\ker(\Gamma_1)$, respectively. In particular, both are found to have simple spectra, a fact that is a direct consequence of the boundary triple construction but does not seem to be easily found in the literature.  

Other self-adjoint extensions are in one-to-one correspondence with self-adjoint linear relations $\ta$ in $\CC^2$, and the rest of Section \ref{s-twolc} focuses on clarifying how and why double eigenvalues arise in these extensions. Subsection \ref{ss-matrixmult2} assumes that $\ta$ is a Hermitian matrix (or, equivalently, that $\dim\ta\ti{mul}=0$) and finds explicit degeneracy conditions similar to those of \cite{BEZ2} that create double eigenvalues within union of resolvent sets $\rho({\bf L}_0)\cup\rho({\bf L}_\infty)$. The authors expect that this union is often all real numbers $\la\geq K$.

The real advantage of this approach is that the conditions on $\ta$ are strict enough that it gives some basic criteria to tell whether a given $\ta$ will produce double eigenvalues or not, with no additional information about solutions. To elaborate, the conditions of Theorem \ref{t-degen} and \cite[Theorem 3]{BEZ2} (Theorem \ref{t-olddouble} here) both involve setting the parameter $\ta$ to take specific values of quasi-derivatives of solutions. However, in practice, if one is given a fixed $\ta$, it can be very difficult to check that it does not meet these conditions for any value of $\la$ (which would then be a double eigenvalue). A consequence of Theorem \ref{t-degen} is that only real, invertible matrices with non-zero off-diagonal entries may achieve a double eigenvalue. If the given matrix $\ta$ does not meet this, the spectrum is simple. 

Herein lies one of the main advantages of boundary triples, it offers additional levels of distinction between parameters and many conditions can be refined. One disadvantage is that we are often restricted to $\rho({\bf L}_0)\cup\rho({\bf L}_\infty)$, which is where the two natural parametrizations are valid, and will sometimes have to deal with $\ta$ that is a self-adjoint linear relation in $\CC^2$ and not a matrix. This case is analyzed in Subsection \ref{ss-LR}. Essentially, such linear relations $\ta$ with $\dim\ta\ti{mul}=1$ (the case $\dim\ta\ti{mul}=2$ simply refers to ${\bf L}_\infty$) can create double eigenvalues, but they can also be achieved by a suitable matrix $\ta$. The conditions on this suitable matrix $\ta$ can be difficult to write explicitly, but we conclude that analyzing matrix $\ta$ is sufficient to determine the full range of spectral properties exhibited by all self-adjoint extensions collectively.

Perturbation theory has always been intimately related with Sturm--Liouville theory, see e.g.~\cite{FL2,Rio}; the original works of Aronszajn \cite{Aron} and Donoghue \cite{Dono} were both inspired by connections to Sturm--Liouville theory. However, there have also been many recent advances in perturbation theory (see e.g.~\cite{LMT21, LT16, LT19, LT22, LT09}), some of which we now apply to Sturm--Liouville operators. First, perturbation theory was not easily applicable to Sturm--Liouville operators with limit-circle endpoints due to the singular behavior causing the additive perturbation to not be well-defined. This was rectified in \cite{BFL} by applying a boundary pair/triple construction to the singular perturbation. Perturbation theoretic techniques were then applied to self-adjoint extensions of the classical Jacobi expression to illustrate new types of possible results with the rank-two perturbation setup, see also \cite{FL}. Extension theory for powers of the derivative and the corresponding matrix-valued spectral theory can be found in \cite{BLM22} using techniques indicated in \cite{AMR}.

Second, the Aronzsajn--Donoghue Theorem, see e.g.~\cite{S}, says that for distinct perturbation parameters the spectral measures of cyclic rank one perturbations by a cyclic vector are mutually singular in the measure theoretic sense. Simple matrix examples, reveal that such mutual singularity does not hold for the scalar (traces of the matrix-valued) spectral measures of higher rank perturbations. However, an appropriate interpretation using the so-called vector mutual singularity of the matrix-valued spectral measures allows for an analogue to the Aronzsajn--Donoghue Theorem for higher rank perturbations, see \cite{LT_JST}. In Section \ref{s-AD} we discuss some applications and results related to Sturm--Liouville operators of several further Aronszajn--Donoghue type theorems such as \cite[Theorem 1.1]{LTV_IMRN} and \cite[Theorem 7.3]{LT_JST}. Motivated by this literature, we are particularly interested in perturbation parameters of the form $\wt\vta+ t\vta$ with $\wt\vta, \vta\in \cM$ and $t\in \R.$ For such parameter families, we also obtain a multiplicity result.

An important distinction between the boundary triples and perturbation theory approaches is that they start from `opposite' self-adjoint extensions: ${\bf L}_0$ is defined by the $\ker(\Gamma_0)$ and therefore is natural for boundary triples, while ${\bf L}_\infty$ has the largest form domain and is necessarily the unperturbed operator for perturbation theory. However, boundary triples can be easily reparameterized to `start' with ${\bf L}_\infty$ and, indeed, it is necessary to use both parametrizations to obtain spectral information in some cases, see Remark \ref{r-alternative} for more details.

Finally, in Section \ref{s-SpecRep} we obtain an expression for the matrix-valued spectral information of the Krein--von Neumann extension (when the operator is positive) of a Sturm--Liouville operator as well as information on the spectral representation for all extensions. The classical Jacobi expression was analyzed in a similar way in \cite{BFL}.

\section{Preliminaries}\label{s-preliminaries}

Consider the classical Sturm--Liouville differential equation
\begin{align}\label{d-sturmdif}
\dfrac{d}{dx}\left[p(x)\dfrac{df}{dx}(x)\right]+q(x)f(x)=-\lambda w(x)f(x),
\end{align}
where $p(x),w(x)>0$ a.e.~on $(a,b)$ and $q(x)$ real-valued a.e.~on $(a,b)$, with $a,b\in\RR\cup\{\pm \infty\}$. 
Furthermore, $1/p(x),q(x),w(x)\in L^1\ti{loc}[(a,b),dx]$. Sturm--Liouville operators are extremely well-studied and there is a vast amount of literature covering their theory and applications; the authors find the book of Zettl \cite{Z}, in particular, very useful and recommend the reader look there to find more information. The sources \cite{AG, BEZ, E, GZ} may also be useful. 

The differential expression can be viewed as a linear operator, mapping a function $f$ to the function $\ell[f]$ via
\begin{align}\label{d-sturmop}
\ell[f](x):=-\dfrac{1}{w(x)}\left(\dfrac{d}{dx}\left[p(x)\dfrac{df}{dx}(x)\right]+q(x)f(x)\right).
\end{align}
This unbounded operator acts on the Hilbert space $L^2[(a,b),w]$, endowed with the inner product 
$
\langle f,g\rangle:=\int_a^b f(x)\overline{g(x)}w(x)dx.
$
In this setting, the eigenvalue problem $\ell[f](x)=\lambda f(x)$ can be considered. However, the operator acting via $\ell[\fdot]$ on $L^2[(a,b),w]$ is not self-adjoint a priori. Additional boundary conditions are required to ensure this property.

\begin{defn}[variation of {\cite[Section 14.2]{N}}]\label{d-defect}
For a symmetric, closed operator ${\bf  A}$ on a Hilbert space $\cH$, define 
the {\bf positive defect space} and the {\bf negative defect space}, respectively, by
$$\cD_+:=\left\{f\in\cD({\bf  A}^*)~:~{\bf  A}^*f=if\right\}
\qquad\text{and}\qquad
\cD_-:=\left\{f\in\cD({\bf  A}^*)~:~{\bf  A}^*f=-if\right\}.$$
\end{defn}

Note that the self-adjoint extensions of a symmetric operator coincide with those of the closure of the symmetric operator {\cite[Theorem XII.4.8]{DS}}, so without loss of generality we assume that all considered operators are closed.

The dimensions dim$(\cD_+)=m_+$ and dim$(\cD_-)=m_-$, called the {\bf positive} and {\bf negative deficiency indices of ${\bf  A}$} respectively, will play an important role. They are usually conveyed as the pair $(m_+,m_-)$. 
The deficiency indices of $T$ correspond to how ``far'' from self-adjoint ${\bf  A}$ is. A symmetric operator ${\bf  A}$ has self-adjoint extensions if and only if its deficiency indices are equal {\cite[Section 14.8.8]{N}}.

\begin{theo}[{\cite[Theorem 14.4.4]{N}}]\label{t-decomp}
If ${\bf  A}$ is a closed, symmetric operator, then the subspaces $\cD({\bf  A})$, $\mathcal{D}_+$, and $\mathcal{D}_{-}$ are linearly independent and their direct sum coincides with $\cD({\bf  A}^*)$, i.e.,
$$\cD({\bf  A}^*)=\cD({\bf  A})\dotplus\mathcal{D}_+ \dotplus\mathcal{D}_{-}.$$
(Here, subspaces $\cX_1, \cX_2, \hdots ,\cX_p$ are said to be {\bf linearly independent}, if $\sum_{i=1}^p x_i = 0$ for $x_i\in \cX_i$ implies that all $x_i=0$.)
\end{theo}

We now let $\ell[\fdot]$ be a Sturm--Liouville differential expression in order to introduce more specific definitions. It is important to reiterate that the analysis of self-adjoint extensions does not involve changing the differential expression associated with the operator at all, merely the domain of definition by applying boundary conditions. 

\begin{defn}[{\cite[Section 17.2]{N}}]\label{d-max}
The {\bf maximal domain} of $\ell[\fdot]$ is given by 
\begin{align*}
\cD\ti{max}=\cD\ti{max}(\ell):=\left\{f:(a,b)\to\mathbb{C}~:~f,pf'\in\text{AC}\ti{loc}(a,b);
f,\ell[f]\in L^2[(a,b),w]\right\}.
\end{align*}
\end{defn}

The designation of ``maximal'' is appropriate in this case because $\cD\ti{max}(\ell)$ is the largest possible subspace that $\ell$ maps back into $L^2[(a,b),w]$. For $f,g\in\cD\ti{max}(\ell)$ and $a<\al\le \beta<b$ the {\bf sesquilinear form} associated with $\ell$ is defined by 
\begin{equation}\label{e-greens}
[f,g]\bigg|_{\al}^{\beta}:=\int_{\al}^{\beta}\left\{\ell[f(x)]\overline{g(x)}-\ell[\overline{g(x)}]f(x)\right\}w(x)dx.
\end{equation}

\begin{theo}[{\cite[Section 17.2]{N}}]\label{t-limits}
The limits $[f,g](b):=\lim_{x\to b^-}[f,g](x)$ and $[f,g](a):=\lim_{x\to a^+}[f,g](x)$ exist and are finite for $f,g\in\cD\ti{max}(\ell)$.
\end{theo}

The equation \eqref{e-greens} is {\bf Green's formula} for $\ell[\fdot]$, and in the case of Sturm--Liouville operators it can be explicitly computed using integration by parts to be the modified Wronskian
\begin{align}\label{e-mwronskian}
[f,g]\bigg|_a^b:=p(x)[f'(x)g(x)-f(x)g'(x)]\bigg|_a^b.
\end{align}

\begin{defn}[{\cite[Section 17.2]{N}}]\label{d-min}
The {\bf minimal domain} of $\ell[\fdot]$ is given by
\begin{align*}
\cD\ti{min}=\cD\ti{min}(\ell):=\left\{f\in\cD\ti{max}(\ell)~:~[f,g]\big|_a^b=0~~\forall g\in\cD\ti{max}(\ell)\right\}.
\end{align*}
\end{defn}

The maximal and minimal operators associated with the expression $\ell[\fdot]$ are then defined as ${\bf L}\ti{min}=\{\ell,\cD\ti{min}\}$ and ${\bf L}\ti{max}=\{\ell,\cD\ti{max}\}$ respectively. By {\cite[Section 17.2]{N}}, these operators are adjoints of one another, i.e.~$({\bf L}\ti{min})^*={\bf L}\ti{max}$ and $({\bf L}\ti{max})^*={\bf L}\ti{min}$.

\begin{defn}\label{d-semibdd}
A symmetric operator ${\bf A}$ is called {\bf semi-bounded (from below)} if for all $f\in\dom({\bf A})$, there exists some $K\in\RR$ such that
\begin{align*}
    \langle{\bf A}f,f\rangle\geq K\langle f,f\rangle.
\end{align*}
\end{defn}

A Sturm--Liouville expression $\ell$ is usually referred to as semi-bounded if ${\bf L}\ti{min}$ satisfies Definition \ref{d-semibdd} for some $K\in\RR$. The important Friedrichs extension will thus also have $K$ as a lower bound, see e.g.~\cite[Proposition 5.3.6]{BdS}. If $K\geq0$, then there also exists a von Neumann--Krein extension, or, if $K<0$ a von Neumann--Krein type extension \cite[Definition 5.4.2]{BdS}. When possible, these extensions will be identified as such. 

We now present a few selected results that can serve as a comparison to the results found in the manuscript. Unfortunately, the notation naturally induced by boundary triples and perturbation theory is very different, so we will introduce this in detail in Section \ref{s-twolc}. Sturm--Liouville literature, when there are two limit-circle endpoints, generally deal with boundary conditions that fall into two disjoint classes: separated and coupled. 

Let boundary condition bases (these are functions in the maximal domain that are linearly independent modulo the minimal domain and normalized within the sesquilinear form) be given by $f,g$ at the endpoint $a$ and $h,k$ at the endpoint $b$. Then, separated conditions, for functions $y\in\cD\ti{max}$, written as
\begin{align*}
    A_1[y,f](a)+A_2[y,g](a)=0 \quad A_1,A_2\in\RR \quad (A_1,A_2)\neq(0,0), \\
    B_1[y,h](b)+B_2[y,k](b)=0 \quad B_1,B_2\in\RR \quad (B_1,B_2)\neq(0,0).
\end{align*}
have the canonical representation 
\begin{align*}
\cos(\al)[y,f](a)-\sin(\al)[y,g](a)=0, \quad 0\leq\al<\pi, \\
\cos(\beta)[y,h](b)-\sin(\beta)[y,k](b)=0, \quad 0<\beta\leq\pi.
\end{align*}
The canonical representation for coupled boundary conditions require 
\begin{align*}
    -\pi<\alpha\leq\pi, \quad R=(r_{ij}), \quad (r_{ij})\in\RR, \quad \det(R)=1,
\end{align*}
and are written
\begin{align}\label{e-connected}
    \begin{pmatrix}
    [y,h](b) \\
    [y,k](b)
    \end{pmatrix}
    =e^{i\alpha}R
    \begin{pmatrix}
    [y,f](a) \\
    [y,g](a)
    \end{pmatrix}.
\end{align}

Now, if $u_1(x,\la)$ and $u_2(x,\la)$ are fundamental solutions to the Sturm--Liouville expression on $(a,b)$ and satisfy the initial conditions 
\begin{align*}
    \begin{pmatrix}
    [u_1,f](a) & [u_2,f](a) \\
    [u_1,g](a) & [u_2,g](a)
    \end{pmatrix}
    =
    \begin{pmatrix}
    1 & 0 \\
    0 & 1
    \end{pmatrix},
\end{align*}
we can define the discriminant
\begin{align}\label{e-discrim}
D(R,\la)=r_{11}[u_2,k](b)+r_{22}[u_1,h](b)-r_{12}[u_2,h](b)-r_{21}[u_1,k](b).
\end{align}

This leads to the following theorem.

\begin{theo}[{\cite[Theorem 1]{BEZ2},\cite[Theorem 10.4.10]{Z}}]\label{t-discrim}
For any $\-\pi\leq\al<\pi$, the number $\la$ is an eigenvalue of the Sturm--Liouville problem $(\ell-\la)y=0$ with boundary condition given by equation \eqref{e-connected} if and only if $D(R,\la)=2\cos(\al)$.
\end{theo}

Eigenvalues of multiplicity two are necessarily more restricted, and conditions can be extracted from the same line of reasoning that proves Theorem \ref{t-discrim}.

\begin{theo}[{\cite[Theorem 3]{BEZ2}}]\label{t-olddouble}
For given $R$ and $-\pi\leq\al<\pi$ a number $\la$ is an eigenvalue of multiplicity two if and only if all four of the following conditions are satisfied:
\begin{align*}
    [u_1,h](b)=e^{i \al} r_{12}, \quad [u_2,h](b)=e^{i\al}r_{11}, \quad [u_1,k](b)=e^{i\al}r_{22}, \quad [u_2,k](b)=e^{i\al}r_{21}.
\end{align*}
\end{theo}

\begin{theo}[{\cite[Theorem 4]{BEZ2}}]\label{t-realsimple}
Let $R$ be as above. Then, for any $\al$ satisfying $-\pi<\al<0$ or $0<\al<\pi$, each eigenvalue of the Sturm--Liouville problem $(\ell-\la)y=0$ with boundary condition given by equation \eqref{e-connected} is simple. 
\end{theo}

Note that the conditions on $\al$ in Theorem \ref{t-realsimple} are exactly those that make the self-adjoint extension of ${\bf L}\ti{min}$ with boundary condition \eqref{e-connected} real (separated boundary conditions also allow for this). To the best knowledge of the authors, these seem to be the most specific results about the multiplicity of general self-adjoint extensions for semi-bounded Sturm--Liouville operators with two limit-circle endpoints.

\subsection{Boundary Triples}\label{ss-bt}

The main tool used for calculating the Weyl $m$-function will be boundary triples. Most of the material from this subsection is taken from the book of Jussi Behrndt, Seppo Hassi, and Henk de Snoo \cite{BdS}, which should be consulted for more details. In particular, we mention that boundary triples can be formulated not only for operators but for more general linear relations. 

\begin{defn}{\cite{BdS}}
Let $\fh$ and $\fK$ be Hilbert spaces over $\CC$. A linear subspace of $\fh\times\fK$ is called a {\bf linear relation} $H$ from $\fh$ to $\fK$ and the elements $\widehat{h}\in H$ will in general be written as pairs $\{h,h'\}$ with components $h\in\fh$ and $h'\in\fK$. If $\fh=\fK$ then we will just say $H$ is a linear relation in $\fh$. 
\end{defn}

\begin{defn}{\cite[Definition 2.1.1]{BdS}}\label{d-bt}
Let $S$ be a closed symmetric relation in a Hilbert space $\fh$. Then $\{\cG,\Gamma_0,\Gamma_1\}$ is a \textbf{boundary triple} for $S^*$ if $\cG$ is a Hilbert space and $\Gamma_0,\Gamma_1:S^*\to\cG$ are linear mappings such that the mapping $\Gamma:S^*\to\cG\times\cG$ defined by
\begin{align*}
    \Gamma\widehat{f}=\{\Gamma_0\widehat{f},\Gamma_1\widehat{f}\}, ~~~~ \widehat{f}=\{f,f'\}\in S^*,
\end{align*}
is surjective and the identity 
\begin{align}\label{e-btgreens}
    \langle f',g\rangle_{\fh}-\langle f,g'\rangle_{\fh}=\langle\Gamma_0\widehat{f},\Gamma_1\widehat{g}\rangle_{\cG}-\langle\Gamma_1\widehat{f},\Gamma_0\widehat{g}\rangle_{\cG},
\end{align}
holds for all $\widehat{f}=\{f,f'\},\widehat{g}=\{g,g'\}\in S^*$.
\end{defn}

Notice that when $S$ is a Sturm--Liouville differential operator the left-hand side of equation \eqref{e-btgreens} is just the sesquilinear form given in Green's formula given by equation \eqref{e-greens}.

The eigenspace of a closed symmetric relation $S$ at $\la\in\CC$ will be written as
\begin{align*}
    \mathfrak{N}_{\la}(S^*)=\ker(S^*-\la) ~~\text{ and we define }~~ \widehat{\mathfrak{N}}_{\la}(S^*)=\left\{\{f_\la,\la f_{\la}\}~:~f_{\la}\in\mathfrak{N}_{\la}(S^*)\right\}.
\end{align*}
Let $\pi_1$ denote the orthogonal projection from $\fh\times\fh$ onto $\fh\times\{0\}$. Then $\pi_1$ maps $\widehat{\mathfrak{N}}_{\la}(S^*)$ bijectively onto $\mathfrak{N}_{\la}(S^*)
\times \{0\}$.

\begin{defn}{\cite[Definition 2.3.1]{BdS}}
Let $S$ be a closed symmetric relation in a complex Hilbert space $\fh$, let $\{\cG,\Gamma_0,\Gamma_1\}$ be a boundary triple for $S^*$, and let $A_0=\ker \Gamma_0$. Then
\begin{align*}
    \rho(A_0)\ni\la\mapsto \gamma(\la)=\left\{\{\Gamma_0\widehat{f}_{\la},f_{\la}\}~:~\widehat{f}_{\la}\in\widehat{\mathfrak{N}}_{\la}(S^*)\right\},
\end{align*}
or, equivalently,
\begin{align*}
    \rho(A_0)\ni\la\mapsto \gamma(\la)=\pi_1\left(\Gamma_0\upharpoonright\widehat{\mathfrak{N}}_{\la}(S^*)\right)^{-1},
\end{align*}
is called the \textbf{$\gamma$-field} associated with the boundary triple $\{\cG,\Gamma_0,\Gamma_1\}$.
\end{defn}

The structure of boundary triples allows for the classical Weyl $M$-function to be obtained via a simple formula.

\begin{defn}{\cite[Definition 2.3.4]{BdS}}\label{d-mfunction}
Let $S$ be a closed symmetric relation in a complex Hilbert space $\fh$, let $\{\cG,\Gamma_0,\Gamma_1\}$ be a boundary triple for $S^*$, and let $A_0=\ker \Gamma_0$. Then
\begin{align*}
    \rho(A_0)\ni\la\mapsto M_0(\la)=\left\{\{\Gamma_0\widehat{f}_{\la},\Gamma_1\widehat{f}_{\la}\}~:~\widehat{f}_{\la}\in\widehat{\mathfrak{N}}_{\la}(S^*)\right\},
\end{align*}
or, equivalently,
\begin{align*}
    \rho(A_0)\ni\la\mapsto M_0(\la)=\Gamma_1\left(\Gamma_0\upharpoonright\widehat{\mathfrak{N}}_{\la}(S^*)\right)^{-1},
\end{align*}
is called the \textbf{Weyl $m$-function} associated with the boundary triple $\{\cG,\Gamma_0,\Gamma_1\}$.
\end{defn}

Let $\{\CC^2,\Gamma_0,\Gamma_1\}$ be a boundary triple for an operator ${\bf A}$ and the deficiency indices of the associated minimal domain $\cD\ti{min}$ be $(2,2)$. Then self-adjoint extensions ${\bf A}(\te)\subset\cD\ti{max}$ are in one-to-one correspondence with the self-adjoint relations $\theta$ in $\CC^2$ via
\begin{align}\label{e-lineardecomp}
    \dom {\bf A}(\te)=\left\{f\in\cD\ti{max}~:~\{\Gamma_0,\Gamma_1\}\in\te\right\}.
\end{align}
According to \cite[Corollary 1.10.9]{BdS}, the relation $\te$ can be represented with two $2\times 2$ matrices $\cA$ and $\cB$ satisfying the conditions $\cA^*\cB=\cB^*\cA$, $\cA\cB^*=\cB\cA^*$ and $\cA\cA^*+\cB\cB^*=I=\cA^*\cA+\cB^*\cB$ such that 
\begin{align*}
    \te=\left\{\{\cA\f,\cB\f\} ~:~ \f\in\CC^n\right\}=\left\{\{\psi,\psi'\} ~:~ \cA^*\psi'=\cB^*\psi\right\}.
\end{align*}
In that case, one has
\begin{align}\label{e-thetadomain1}
\dom( {\bf A}(\ta))=\left\{ f\in\cD\ti{max} ~:~ \cA^*\Gamma_1(f)=\cB^*\Gamma_0(f)\right\}.
\end{align}
Theorem 2.6.1 and Corollary 2.6.3 from \cite{BdS} then say that for $\la\in\rho({\bf A}_{\te})\cap\rho({\bf A}_0)$ the Krein formula for the corresponding resolvents are given by
\begin{equation}\label{e-thetaresolvent}
\begin{aligned}
({\bf A}(\te)-\la)^{-1}&=({\bf A}_0-\la)^{-1}+\gamma(\la)(\te-M_0(\la))^{-1}\gamma(\overline{\la})^* \\
&=({\bf A}_0-\la)^{-1}+\gamma(\la)\cA(\cB-M_0(\la)\cA)^{-1}\gamma(\overline{\la})^*.
\end{aligned}
\end{equation}
In the case of the examples in this manuscript, the spectrum of ${\bf A}_0$ is discrete and the difference of the resolvents of ${\bf A}_0$ and ${\bf A}(\te)$ is an operator of rank $\leq$ $2$. Thus, the spectrum of the self-adjoint operator ${\bf A}(\te)$ is also discrete. Indeed, $\la\in\rho({\bf A}_0)$ is an eigenvalue of ${\bf A}(\te)$ if and only if $\ker (\te-M_0(\la))$ --- or equivalently, $\ker(\cB-M_0(\la)\cA)$ is nontrivial --- and that 
\begin{align*}
    \ker({\bf A}(\te)-\la)=\gamma(\la)\ker(\te-M_0(\la))=\gamma(\la)\cA\ker(\cB-M_0(\la)\cA).
\end{align*}
In the special case that $\te$ is also a Hermitian $2\times 2$ matrix, the boundary condition for the domain of ${\bf A}(\te)$ can be written as
\begin{align}\label{e-thetadomain2}
    \dom ({\bf A}(\ta))=\left\{ f\in\cD\ti{max} ~:~ \te\Gamma_0(f)=\Gamma_1(f)\right\}.
\end{align}
The spectral properties of the operator ${\bf A}(\ta)$ can also be described with the help of the function 
\begin{align}\label{e-thetam}
    \la\mapsto(\te-M_0(\la))^{-1};
\end{align}
the poles of the matrix function \eqref{e-thetam} coincide with the discrete spectrum of ${\bf A}(\te)$ and the dimension of the eigenspace $\ker({\bf A}(\te)-\la)$ coincides with the dimension of the range of the residue of the function \eqref{e-thetam} at $\la$.

\section{Two Limit-Circle Endpoints}\label{s-twolc}

We now let ${\bf L}$ be a semi-bounded Sturm--Liouville operator with lower bound $K$ defined on a subset $(a,b)$ of $\RR\cup\{\pm\infty\}$ with associated maximal domain $\cD\ti{max}$ and sesquilinear form $[\cdot,\cdot]$ such that $a$ and $b$ are both limit-circle non-oscillating. Solutions to the differential expression $\ell$ fall into two categories.

\begin{defn}{\cite[Definition 6.10.3]{BdS}}\label{d-principal}
Let $(\ell-\lambda)f=0$ with $\lambda\in\RR$ be non-oscillatory at the endpoint $a$ and let $u$ and $v$ be real solutions of $(\ell-\lambda)f=0$. Then $u$ is said to be {\bf principal} at $a$ if $1/pu^2$ is not integrable at $a$ and $v$ is said to be {\bf non-principal} at $a$ if $1/pv^2$ is integrable at $a$. This holds analogously at the endpoint $b$.
\end{defn}

A principal solution of $(\ell-\lambda)$, $\lambda\in\RR$, always exists and is unique up to real nonzero multiples. In fact (\cite[Corollary 6.10.5]{BdS}), a nontrivial solution $u$ is principal at $a$ if and only if 
\begin{align*}
    \lim_{x\to a}\dfrac{u(x)}{v(x)}=0,
\end{align*}
for all solutions $v$ of $(\ell-\lambda)y=0$ which are linearly independent of $u$. Additionally, a non-principal solution can be obtained from such a function $u$ by \cite[Theorem 6.10.4]{BdS} as follows. Assume that $u$ does not vanish on $(a,a_0)$ and let $\al\in(a,a_0)$. Then a function $v$ is a real non-principal solution of $(\ell-\lambda)y=0$ with $[u,v](a)=1$ if and only if there exists $\beta\in\RR$ such that 
\begin{align*}
    v(x)=-u(x)\left(\beta+\int_x^{\al}\dfrac{ds}{p(s)u(s)^2}\right).
\end{align*}

Fix $\la_0\in\RR$. Let $u_a$ and $v_a$ be principal and non-principal solutions to $(\ell-\la_0)y=0$, respectively, at $a$ such that $[u_a,v_a](a)=1$ and $u_b$ and $v_b$ satisfy an analogous condition at $b$. We define two $C^2$ functions that behave as follow near the endpoints
\begin{align*}
    \widetilde{u}(x,\la_0):=
    \begin{cases}
          u_a(x,\la_0) \text{ near }a \\
          u_b(x,\la_0) \text{ near }b \\
    \end{cases}\hspace{-1.1em}\Bigg\},
    \hspace{2em}
    \widetilde{v}(x,\la_0):=
    \begin{cases}
          v_a(x,\la_0) \text{ near }a \\
          v_b(x,\la_0) \text{ near }b \\
    \end{cases}\hspace{-1.1em}\Bigg\}.
\end{align*}

In particular, we have that $[\widetilde{u},\widetilde{v}](a)=[\widetilde{u},\widetilde{v}](b)=1$. Quasi-derivatives can then be defined as the pairing of functions $f\in\cD\ti{max}$ with $\widetilde{u}$ and $\widetilde{v}$ via \cite[Theorem 4.5]{GLN}:

\begin{align*}
f^{[0]}(a)&=[f,\widetilde{v}](a)=[f,v_a](a)=\lim_{x\to a}\dfrac{f(x)-f^{[1]}(a)v_a(x,\la_0)}{u_a(x,\la_0)}, \\
 f^{[0]}(b)&=[f,\widetilde{v}](b)=[f,v_b](b)=\lim_{x\to b}\dfrac{f(x)-f^{[1]}(b)v_b(x,\la_0)}{u_b(x,\la_0)}, \\
 f^{[1]}(a)&=-[f,\widetilde{u}](a)=-[f,u_a](a)=\lim_{x\to a}-\dfrac{f(x)}{v_a(x,\la_0)}, \\
  f^{[1]}(b)&=-[f,\widetilde{u}](b)=-[f,u_b](b)=\lim_{x\to b}-\dfrac{f(x)}{v_b(x,\la_0)}.
\end{align*}

Recall that for all $f\in\cD\ti{max}$ the quasi-derivatives $f^{[0]}(a)$, $f^{[1]}(a)$, $f^{[0]}(b)$ and $f^{[1]}(b)$ are well-defined due to Theorem \ref{t-limits}.
Fix a fundamental system $(u_1(\cdot,\la);~u_2(\cdot,\la))$ for the equation $(\ell-\la)y=0$ by the initial conditions
\begin{align}\label{e-ic}
    \left( \begin{array}{cc}
u_1^{[0]}(a,\la) &  u_2^{[0]}(a,\la) \\
u_1^{[1]}(a,\la) & u_2^{[1]}(a,\la)
\end{array} \right)
=
    \left( \begin{array}{cc}
1 & 0 \\
0 & 1
\end{array} \right).
\end{align}
Without loss of generality, we also assume that the solutions $u_1$ and $u_2$ are Wronskian normalized at $x=b$, i.e. $[u_1,u_2](b)=1$ for all $\la$. Indeed, if $[u_1,u_2](b)=0$ for some $\la$ then $u_1$ and $u_2$ are not linearly independent solutions at $\la$, a contradiction, and if $[u_1,u_2](b)\neq0$ we just renormalize.

\begin{prop}{\cite[Proposition 6.3.8]{BdS}}\label{p-l0}
Assume that the endpoints $a$ and $b$ are in the limit-circle case. Then $\{\CC^2,\Gamma_0,\Gamma_1\}$, where
\begin{align}
    \Gamma_0f:=\left(\begin{array}{c}
f^{[0]}(a)  \\
f^{[0]}(b)
\end{array} \right), \hspace{.5cm}
    \Gamma_1 f:=\left( \begin{array}{c}
f^{[1]}(a)   \\
-f^{[1]}(b) 
\end{array} \right), \hspace{.5cm}
f\in\dom T\ti{max},
\end{align}
is a boundary triple for $\cD\ti{max}$. The self-adjoint extension ${\bf L}_0$ corresponding to $\Gamma_0$ is the restriction of ${\bf L}\ti{max}$ defined on 
\begin{align}\label{e-l0}
    \dom {\bf L}_0=\{f\in\dom \cD\ti{max} ~:~ f^{[0]}(a)=f^{[0]}(b)=0 \},
\end{align}
and $u_2^{[0]}(b,\lambda)\neq 0$ for all $\lambda\in\rho({\bf L}_0)$. Moreover, the corresponding Weyl function is
\begin{align*}
    M_0(\lambda)=\dfrac{1}{u_2^{[0]}(b,\lambda)}
    \left( \begin{array}{cc}
-u_1^{[0]}(b,\lambda) & 1 \\
1 & -u_2^{[1]}(b,\lambda)
\end{array} \right),
\end{align*}
for $\lambda\in\rho({\bf L}_0)$.
\end{prop}

Switching the labels of $\Gamma_0$ and $\Gamma_1$ and multiplying one of the maps by $-1$ preserves the boundary triple while giving results for another important extension, see \cite[Corollary 2.5.4]{BdS}.

\begin{cor}\label{c-l1}
Assume that the endpoints $a$ and $b$ are in the limit-circle case. Then $\{\CC^2,\Gamma_0',\Gamma_1'\}$, where
\begin{align*}
    \Gamma_0'f:=\left(\begin{array}{c}
-f^{[1]}(a)  \\
f^{[1]}(b)
\end{array} \right), \hspace{.5cm}
    \Gamma_1'f:=\left( \begin{array}{c}
f^{[0]}(a)   \\
f^{[0]}(b) 
\end{array} \right), \hspace{.5cm}
f\in\dom T\ti{max},
\end{align*}
is a boundary triple for $\cD\ti{max}$.
The self-adjoint extension ${\bf L}_\infty$ corresponding to $\Gamma_1$ is the restriction of ${\bf L}\ti{max}$ defined on 
\begin{align}\label{e-l1}
    \dom {\bf L}_\infty=\{f\in\dom \cD\ti{max} ~:~ f^{[1]}(a)=f^{[1]}(b)=0 \},
\end{align}
and $u_1^{[1]}(b,\lambda)\neq 0$ for all $\lambda\in\rho({\bf L}_\infty)$. Moreover, the corresponding Weyl function is given by 
\begin{align*}
    M_\infty(\lambda)=\dfrac{1}{u_1^{[1]}(b,\lambda)}
    \left( \begin{array}{cc}
u_2^{[1]}(b,\lambda) & 1 \\
1 & u_1^{[0]}(b,\lambda)
\end{array} \right),
\end{align*}
for $\lambda\in\rho({\bf L}_\infty)$.
\end{cor}

The $m$-functions of these extensions will be critical for the spectral analysis of all self-adjoint extensions. Let $\ta$ be the self-adjoint linear relation that determines the self-adjoint extension, as in equation \eqref{e-thetadomain1} or, when $\ta$ is a Hermitian $2\times 2$ matrix, equation \eqref{e-thetadomain2}. Such a linear relation can be used to calculate the $m$-function for the extension ${\bf L}(\ta)$. In particular, if $\ta$ is a $2\times 2$ Hermitian matrix and $\la\in\rho({\bf L}_0)\cap\rho({\bf L}(\ta))$, \cite[Equation (3.8.7)]{BdS} says that the $m$-function of the extension ${\bf L}(\ta)$ is 
\begin{align}\label{e-MTheta}
M(\ta,\lambda)=(\ta-M_0(\lambda))^{-1}.
\end{align}

With some algebra and the normalizations $[u_1,u_2](a)=[u_1,u_2](b)=1$ it is easily verified that plugging $\ta=\left\{\CC^2,0\right\}$, the $0$ matrix, which corresponds to the extension ${\bf L}_\infty$, into equation \eqref{e-MTheta} yields $-M_0^{-1}=M_\infty$.

\begin{rem}\label{r-alternative}
Of course, it is also possible to rewrite equation \eqref{e-MTheta} around $M_\infty$ using the boundary triple in Corollary \ref{c-l1} where we recover
\begin{align}\label{e-wtm}
\wt{M}(\vta,\la)=(\vta-M_\infty(\la))^{-1}
\end{align}
and $M_0=-M_\infty^{-1}$, for $\vta$ not equal to $\ta$ in general. The $\wt{M}$ notation is used to clarify the relationship between $m$-functions, as evaluations are often performed in terms of both $\ta$ and $\vta$. This re-parametrization around ${\bf L}_\infty$ can be especially advantageous to deal with $\la\in\sigma({\bf L}_0)$, a blind spot of equation \eqref{e-MTheta}, and will be play an important role later. Importantly, this means that $\vta=-\ta^{-1}$ and $\vta$ still realizes all possible self-adjoint restrictions of ${\bf L}\ti{max}$ via an analog of equation \eqref{e-thetadomain1}. The notation $\wt{\bf L}(\vta)$ will be used to distinguish these extensions from ${\bf L}(\ta)$.
\hfill$\diamondsuit$
\end{rem}

We have observed that knowing the $m$-function for either extension ${\bf L}_0$ or ${\bf L}_\infty$ allows for the easy calculation of the other, but we prefer to center our calculations around ${\bf L}_0$ when possible. However, utilizing information from both $m$-functions will allow for a more detailed spectral analysis. In particular, we recall the following proposition which will be able to help determine the multiplicity of eigenvalues via the $m$-functions, a focus of our spectral analysis. Note we abbreviate the classes of $2\times 2$ Hermitian matrices as $\cM$ and of self-adjoint linear relations in $\CC^2$ as $\cR$ so that $\cM\subset\cR$.

\begin{prop}{\cite[Proposition 1]{DM2}}\label{p-dmmult}
Let $\{\CC^2,\Gamma_0,\Gamma_1\}$ be the boundary triple defined in Proposition \ref{p-l0}, $\ta\in\cR$ and  $\la\in\rho({\bf L}_0)$. Then
\begin{align*}
    \la\in\sigma_p({\bf L}(\ta)) \iff 0\in\sigma_p(\ta-M_0(\la)).
\end{align*}
Moreover, we have $\dim\ker({\bf L}(\ta)-\la)=\dim\ker(\ta-M_0(\la))$. Alternatively, let $\{\CC^2,\Gamma_0',\Gamma_1'\}$ be the boundary triple defined in Corollary \ref{c-l1}, $\vta\in\cR$ and $\la\in\rho({\bf L}_\infty)$. Then
\begin{align*}
    \la\in\sigma_p(\wt{\bf L}(\vta)) \iff 0\in\sigma_p(\vta-M_\infty(\la)).
\end{align*}
Moreover, we have $\dim\ker(\wt{\bf L}(\vta)-\la)=\dim\ker(\vta-M_\infty(\la))$.
\end{prop}

The portion of Proposition \ref{p-dmmult} using $\vta$ can easily be translated into the portion using $\ta$ using the conversion $\ta=-\vta^{-1}$ and tracking which parameter refers to ${\bf L}_0$ and ${\bf L}_\infty$. Also, note that the algebraic and geometric multiplicities for Sturm--Liouville operators with limit-circle endpoints are the same, see \cite{KWZ}, and so we simply refer to the shared value as multiplicity here.

We start our explicit analysis with $M_0$, which determines the spectrum of ${\bf L}_0$. Recall that ${\bf L}_0$ is also known as the Krein--von Neumann extension of ${\bf L}\ti{min}$ when the underlying operator $\ell$ is not just bounded from below, but is also positive. In the case when $\ell$ is not positive the extension is still crucial, as it is intimately connected with the Friedrichs extension ${\bf L}_\infty$ due to their inherent transversality. It also allows for the construction of the so-called Krein--von Neumann type extension when a shift is applied, see \cite[Section 5.4]{BdS}. We will discuss this fact further after the proposition.

\begin{prop}\label{p-kvn}
Let ${\bf L}_0$ be as above with domain given by equation \eqref{e-l0} and $\la\in\CC$. Then the following statements are equivalent:
\begin{enumerate}
    \item[(i)] $\la\in\sigma_d({\bf L}_0)=\sigma({\bf L}_0)$, where $\sigma_d$ denotes the discrete spectrum. 
    \item[(ii)] We have $u_2^{[0]}(b,\la)=0$ and hence $u_2(x,\la)\in\dom({\bf L}_0)$.
    \item[(iii)] We have
    \begin{align*}
   \langle u_2(x,\la),\widetilde{v}(x,\la_0)\rangle=0 \text{ in } L^2[(a,b),wdx],
    \end{align*}
    and hence $u_2(x,\la)$ is orthogonal to $\widetilde{v}(x,\la_0)$ for such $\la$, or $\la=\la_0$.
    \item[(iv)] If $v_b(x,\la_0)$ does not vanish on $(d,b)$ for some $a<d<b$ and $\beta$ is such that $d<\beta<b$, then 
    \begin{align*}
        -u_2^{[0]}(\beta,\la)=(\la-\la_0)\int_{\beta}^bu_2(x',\la)v_b(x',\la_0)w(x')dx'.
    \end{align*}
    \item[(v)] If either both $a$ and $b$ are regular singular points of $\ell[\cdot]$ or there exists a M\"obius transformation $f(x):(a,b)\to(a,b)$ such that $f(a)=b$ and $u_2(f(x),\la)$ is also a solution to $(\ell-\la)y=0$, then
    \begin{align*}
        u_2(f(x),\la)=u_2^{[1]}(b,\la)u_2(x,\la).
    \end{align*}
\end{enumerate}
Furthermore, each $\la\in\sigma({\bf L}_0)$ has multiplicity one.
\end{prop}

\begin{proof}
First note that the statement $u_2(x,\la)\in\dom({\bf L}_0)$ in $(ii)$ follows from observing that the initial condition that $u_2^{[0]}(a,\la)=0$ combined with Theorem \ref{t-limits} and $u_2^{[0]}(b,\la)=0$ means that $u_2^{[0]}(x,\la)\in\dom{\bf L}_0$ for such a $\la$.\\

$(i)\iff(ii)$:
The spectrum of ${\bf L}_0$ is discrete because both endpoints are limit-circle. Eigenvalues occur at poles of $M_0(\lambda)$. All of the quasi-derivatives in the entries of $M_0$ exist and are finite by Theorem \ref{t-limits}, so poles occur if and only if $u_2^{[0]}(b,\la)=0$. \\

$(ii)\iff(iii)$: Using Green's identity \eqref{e-greens} and the modified Wronskian \eqref{e-mwronskian}, we see that
\begin{align}\label{equationx}
    (\la-\la_0)\int_a^bu_2\widetilde{v}wdx&=\int_a^b({\bf L}\ti{max}u_2)\overline{\widetilde{v}}wdx-\int_a^bu_2(\overline{{\bf L}\ti{max}\widetilde{v}})wdx \\
    &=[u_2(x,\la),\widetilde{v}(x,\la_0)]\Big|_{x=a}^b 
\\
    &=u_2^{[0]}(b,\la)-u_2^{[0]}(a,\la).\label{equationz}
\end{align}
The last equality follows from the the normalization of $\tilde{v}$. 
The equivalence of the two statements follows from the initial conditions in equation \eqref{e-ic} as $u_2^{[0]}(a,\la)=0$. \\

$(ii)\iff(iv)$: 
In equations \eqref{equationx} through \eqref{equationz}, we replace the lower bound $x=a$ by $x=\beta$, for $\beta$ as in the statement of $(iv)$ and $\tilde v$ by $v_b$. This easily yields the `localized' result. \\

$(ii)\iff(v)$: First assume that both $a$ and $b$ are regular singular points of $\ell[\cdot]$. Then there exists a change of variables so that $\ell$ can be written as the classical hypergeometric equation, see i.e.~\cite{BFL} and \cite[Section 15.10]{DLMF}. Connection formulas \cite[Section 15.10]{DLMF} for solutions to this equation exist so that it is possible to write 
\begin{align}\label{e-connect}
    u_2(f(x),\la)=c_1(\la)u_1(x,\la)+c_2(\la)u_2(x,\la),
\end{align}
for $f(x)$ a M\"obius transformation such that $f(a)=b$ (the explicit $f(x)$ depends on which endpoints are infinite and the previous change of variables). Taking the limit as $x\to a^+$ and then applying quasiderivatives, allows us to see that $u_2^{[0]}(b,\la)=c_1(\la)$ and $u_2^{[1]}(b,\la)=c_2(\la)$ thanks to the initial conditions in equation \eqref{e-ic}. Condition $(ii)$ says that $c_1(\la)=0$ and the result follows from plugging into equation \eqref{e-connect}. 

Likewise, if there exists a M\"obius transformation $f(x):(a,b)\to(a,b)$ such that $f(a)=b$ and $u_2(f(x),\la)$ is a solution to $(\ell-\la)y=0$, then $u_2(f(x),\la)$ can be written as a linear combination in $x$ of $u_1(x,\la)$ and $u_2(x,\la)$ by definition, as in equation \eqref{e-connect}. Condition $(v)$ then follows in the same way as when using the other hypothesis.

For the converse, assume that there exists some M\"obius transformation, which may stem from both $a$ and $b$ being regular singular points, as above, $f(x):(a,b)\to(a,b)$ such that $f(a)=b$ and $u_2((f(x),\la)$ is a solution to $(\ell-\la)y=0$. Since by (v) we have that $u_2((f(x),\la)=u_2^{[1]}(b,\la)u_2(x,\la)$, taking the limit as $x\to a^+$ and then applying the $[0]$ quasiderivative yields
\begin{align*}
u_2^{[0]}(b,\la)=u_2^{[1]}(b,\la)u_2^{[0]}(a,\la).
\end{align*}
The initial condition that $u_2^{[0]}(a,\la)=0$ then implies that $u_2^{[0]}(b,\la)=0$.\\

In order to show that the multiplicity of each eigenvalue is one, consider the boundary triple $\{\CC^2,\Gamma_0',\Gamma_1'\}$ from Corollary \ref{c-l1}. Setting $\vta=0$ then yields $\wt{\bf L}(\vta)={\bf L}_0$ in Proposition \ref{p-dmmult} and we find that the multiplicity of each eigenvalue $\la$ is given by $\dim\ker(\vta-M_\infty(\la))$. It is thus enough to prove that $M_\infty(\la)$ does not have rank 0 or, equivalently, that it does not have nullity 2. Upon inspection, this occurs if and only if $1/u_1^{[1]}(b,\la)=0$. This is impossible, as $u_1^{[1]}(b,\la)$ exists and is finite for all $\la$ by Theorem \ref{t-limits}. The result follows.
\end{proof}

Before proceeding, we will discuss some of the consequences of Proposition \ref{p-kvn}. First, condition $(iii)$ of Proposition \ref{p-kvn} is a bit surprising. While it is clear that ${\bf L}_0$ has compact resolvent and so we expect orthogonality of eigenfunctions, the statement is still unexpected in the sense that functions in ${\bf L}_0$, including these eigenfunctions, should have the same asymptotic behavior as $\widetilde{v}$ near the endpoints in order to meet the boundary conditions. The function $\wt{v}$ is fixed but there is freedom in its behavior away from the endpoints.

Condition $(v)$ is perhaps the most unexpected: it says the asymptotic behavior of the solution $u_2(x,\la)$ must be the same at each endpoint for the specified values of $\la$. Connection formulas like the kind in equation \eqref{e-connect} were crucial in determining the eigenvalues for extensions of the Jacobi differential operator in \cite{F}, but their existence when $a$ and $b$ are not regular singular points is unclear. The special case where $u_2(f(x),\la)$ is a solution to $(\ell-\la)f=0$ seems to be an outlier. Still, the fact that eigenvalues of extensions are embedded into such connection formulas for solutions is intriguing.

Finally, the multiplicity of $\la\in\sigma({\bf L}_0)$ depending on the rank of the Weyl-$m$ function for the transversal extension ${\bf L}_\infty$ is extremely convenient, as this is the only other extension which is very accessible. The interdependence of these extensions seems to be rarely mentioned in the literature for limit-circle endpoints but is of key importance for inverse spectral theory when there are regular endpoints, see e.g.~\cite{CGNZ,Eck,EGNT2,Burak}. Determining the multiplicity of eigenvalues for extensions other than these two will prove to be more difficult.

It is possible to prove an analogous proposition concerning the eigenvalues of the important transversal extension ${\bf L}_\infty$. In the case where $\ell$ is bounded from below, ${\bf L}_\infty$ is known as the Friedrichs extension, see i.e.~\cite{MZ}.

\begin{prop}\label{p-f}
Let ${\bf L}_\infty$ be as above with domain given by equation \eqref{e-l1} and $\la\in\CC$. Then the following statements are equivalent:
\begin{enumerate}
    \item[(i)] $\la\in\sigma_d({\bf L}_\infty)=\sigma({\bf L}_\infty)$.
    \item[(ii)] We have $u_1^{[1]}(b,\la)=0$ and hence $u_1(x,\la)\in\dom({\bf L}_\infty)$.
    \item[(iii)] We have
    \begin{align*}
   \langle u_1(x,\la),\widetilde{u}(x,\la_0)\rangle=0 \text{ in } L^2[(a,b),wdx].
    \end{align*}
    and hence $u_1(x,\la)$ is orthogonal to $\widetilde{u}(x,\la_0)$ for such $\la$, or $\la=\la_0$.
    \item[(iv)] If $u_b(x,\la_0)$ does not vanish on $(d,b)$ for some $a<d<b$ and $\beta$ is such that $d<\beta<b$, then 
    \begin{align*}
        u_1^{[1]}(\beta,\la)=(\la-\la_0)\int_{\beta}^bu_1(x',\la)u_b(x',\la_0)w(x')dx'.
    \end{align*}
     \item[(v)] If either both $a$ and $b$ are regular singular points of $\ell[\cdot]$ or there exists a M\"obius transformation $f(x):(a,b)\to(a,b)$ such that $f(a)=b$ and $u_2(f(x),\la)$ is also a solution to $(\ell-\la)y=0$, then
    \begin{align*}
        u_1(f(x),\la)=u_1^{[0]}(b,\la)u_1(a,\la).
    \end{align*}
\end{enumerate}
Furthermore, each $\la\in\sigma({\bf L}_\infty)$ has multiplicity one.
\end{prop}

\begin{proof}
The proof is analogous to the proof of Proposition \ref{p-kvn}. Note that to show the multiplicity of $\la\in\sigma({\bf L}_\infty)$ is one it is necessary to analyze $M_0(\la)$.
\end{proof}

It is possible for $\sigma({\bf L}_0)\cap\sigma({\bf L}_\infty)\neq \varnothing$. For $\la$ in this intersection, Propositions \ref{p-kvn} and \ref{p-f} imply that $u_1^{[1]}(b,\la)=u_2^{[0]}(b,\la)=0$. However, this does not necessarily mean that there exists a self-adjoint extension ${\bf L}(\ta_0)$, with $\ta_0\notin\{0\}\cup\{\infty\}$, that has $\la$ as an eigenvalue, or the multiplicity of this eigenvalue. However, some basic criterion is available.

\begin{cor}\label{c-basicint}
Let $\la\in\sigma({\bf L}_0)\cup\sigma({\bf L}_\infty)$. Then $u_1^{[0]}(b,\la)u_2^{[1]}(b,\la)=1.$
\end{cor}

\begin{proof}
The fundamental solutions are normalized within the Wronskian by definition and satisfy $[u_1,u_2](b)=1$ for all $\la$. If $\la\in\sigma({\bf L}_0)\cup\sigma({\bf L}_\infty),$ then either $u_1^{[1]}(b,\la)$ or $u_2^{[0]}(b,\la)$ is 0, so we calculate
\begin{align*}
    1&=[u_1(x,\la),u_2(x,\la)](b)=u_1^{[0]}(b,\la)u_2^{[1]}(b,\la)-u_1^{[1]}(b,\la)u_2^{[0]}(b,\la) \\
    &=u_1^{[0]}(b,\la)u_2^{[1]}(b,\la),
\end{align*}
and the result follows. 
\end{proof}

\begin{prop}\label{p-varnothing}
For every $\ta_0\in\cR$ there exists $\ta_1\in\cR$ with $\sigma({\bf L}(\ta_0))\cap \sigma({\bf L}(\ta_1))=\varnothing.$
\end{prop}

The proof of this proposition is deferred until after Remark \ref{r-after}, as notation from Section \ref{s-AD} is essential.

\subsection{Spectral properties for other self-adjoint extensions}\label{ss-matrixmult2}

This knowledge of the spectral properties of the self-adjoint extensions ${\bf L}_0$ and ${\bf L}_\infty$ now allows us to pass to other general self-adjoint extensions. To this end, let $\ta\in\cM$ and write
\begin{equation*}
\ta= 
\begin{pmatrix} 
\ta_{11} & \ta_{12}\\
\ta_{21} & \ta_{22}
\end{pmatrix}.
\end{equation*}
Plugging into equation \eqref{e-MTheta} yields
\begin{align}
\label{e-weyltheta}
    M(\ta,\lambda)&=(\ta - M_0(\lambda))^{-1}
=\begin{pmatrix}
\ta_{11}+\frac{u_1^{[0]}(b,\lambda)}{u_2^{[0]}(b,\lambda)} & \ta_{12}-\frac{1}{u_2^{[0]}(b,\lambda)} \\
\ta_{21}-\frac{1}{u_2^{[0]}(b,\lambda)} & \ta_{22}+\frac{u_2^{[1]}(b,\lambda)}{u_2^{[0]}(b,\lambda)}
\end{pmatrix}^{-1}.
\end{align}

A general analog of equation \eqref{e-MTheta} can also be written for the case where $\ta\in\cR$ (see \cite[Equation 3.8.4]{BdS}) which relies on the decomposition from equation \eqref{e-lineardecomp}:
 \begin{align}\label{e-Mlinearrelation}
 M(\ta,\la)=\left(\cA^*+\cB^*M_0(\la)\right)\left(\cB^*-\cA^*M_0(\la)\right)^{-1}.
 \end{align}

To further discuss the eigenvalues coming from different choices of $\ta$, we recall that we abbreviate the classes of $2\times 2$ Hermitian matrices as $\cM$ and of self-adjoint linear relations in $\CC^2$ as $\cR$ so that $\cM\subset\cR$. Our analysis of the eigenvalues of $M(\ta)$ begins in the special case where $\ta\in\cM$. Also recall that $\vta=-\ta^{-1}$, which naturally arises from Remark \ref{r-alternative}, and that the inverse of a linear relation always exists. These parametrizations are made without loss of generality; a general self-adjoint extension of ${\bf L}\ti{min}$ can always be transcribed as ${\bf L}(\ta)$ for some $\ta\in\cR$, and then $\vta$ can be calculated to yield $\widetilde{\bf L}(\vta)$. In this subsection, we are mainly concerned with the case $\ta\in\cM$ and, in particular, eigenvalues of multiplicity two. A discussion of how eigenvalues of multiplicity two can arise when $\ta\in\cR\backslash\cM$ is postponed to Subsection \ref{ss-LR}.
 
\begin{defn}\label{d-degenerate}
An eigenvalue $\la\in\RR$ of a self-adjoint extension ${\bf L}(\ta)$ (or $\widetilde{\bf L}(\vta)$) will be said to be {\bf degenerate} if it satisfies one of the following conditions:
\begin{itemize}
    \item[(i)] If $\la\in\rho({\bf L}_0)$, $\ta\in\cM$ and
    \begin{align*}
        \ta_{11}=-\frac{u_1^{[0]}(b,\lambda\ci\ta)}{u_2^{[0]}(b,\lambda\ci\ta)}, \quad \ta_{21}=\frac{1}{u_2^{[0]}(b,\lambda\ci\ta)},\quad\text{and}\quad \ta_{22}=-\frac{u_2^{[1]}(b,\la\ci\ta)}{u_2^{[0]}(b,\la\ci\ta)}.
    \end{align*}
    \item[(ii)] If $\la\in\rho({\bf L}_\infty)$, $\vta\in\cM$ and
    \begin{align*}
        \vta_{11}=\frac{u_2^{[1]}(b,\lambda\ci\ta)}{u_1^{[1]}(b,\lambda\ci\ta)}, \quad \vta_{21}=\frac{1}{u_1^{[1]}(b,\la\ci\ta)} \quad\text{and} \quad \vta_{22}=\frac{u_1^{[0]}(b,\la\ci\ta)}{u_1^{[1]}(b,\lambda\ci\ta)}.
    \end{align*}
    \item[(iii)] If $\la\in\sigma({\bf L}_0)\cap\sigma({\bf L}_\infty)$ and $\la$ has multiplicity two for ${\bf L}(\ta)$ (or $\widetilde{\bf L}(\vta)$). 
\end{itemize}
\end{defn}

In Cases $(i)$ and $(ii)$, each entry of $\ta$ and $\vta$ is well-defined because $u_2^{[0]}(b,\la\ci\ta)\neq0$ and $u_1^{[1]}(b,\la\ci\ta)\neq0$, respectively.

Degenerate eigenvalues thus correspond to the case where $\ta-M_0$ (or $\vta-M_\infty$) is the zero matrix in equation \eqref{e-weyltheta}. An operator ${\bf L}(\ta)$ will be called non-degenerate if all of its eigenvalues are non-degenerate. Also recall that $K\in\RR$ is assumed to be the lower bound of the symmetric operator ${\bf L}\ti{min}$ via Definition \ref{d-semibdd}.

\begin{theo}\label{t-degen}
If $\la\ci\ta\in\RR$ is a degenerate eigenvalue of ${\bf L}(\ta)$ or  $\widetilde{\bf L}(\vta)$, then $\la\ci\ta$ has multiplicity two. Conversely, if $K\leq\la\ci\ta\in\rho({\bf L}_0)$ and $\ta\in\cM$ or $K\leq\la\ci\ta\in\rho({\bf L}_\infty)$ and $\vta\in\cM$ has multiplicity two, then $\la\ci\ta$ is degenerate. 
\end{theo}


\begin{proof}
Assume that $\la\ci\ta$ is a degenerate eigenvalue of ${\bf L}(\ta)$. For $\la\ci\ta\in\rho({\bf L}_0)$ the above observation that degenerate eigenvalues correspond to the case where $\ta-M_0$ yields the zero matrix and Proposition \ref{p-dmmult} immediately imply that $\la\ci\ta$ must be an eigenvalue of multiplicity two. For $\la\ci\ta\in\sigma({\bf L}_0)$ but $\la\ci\ta\in\rho({\bf L}_\infty)$, the same argument holds except $\vta-M_\infty$ yields the zero matrix. For $\la\ci\ta\in\sigma({\bf L}_0)\cap\sigma({\bf L}_\infty)$, $\la\ci\ta$ has multiplicity two by the definition of degeneracy. 

The proof of the converse implication will be shown only for the case where $\la\ci\ta\in\rho({\bf L}_0)$. The case where $\la\ci\ta\in\sigma({\bf L}_0)$ but $\la\ci\ta\in\rho({\bf L}_\infty)$ follows by a similar argument and the case where $\la\ci\ta\in\sigma({\bf L}_0)\cap\sigma({\bf L}_\infty)$ is immediate. 

Let $\ta$ be a Hermitian $2\times 2$ matrix and $\la\ci\ta$ be an eigenvalue of ${\bf L}(\ta)$ with multiplicity two. If we omit the dependence of our solutions on $x=b$ and of $\la$ on $\ta$ momentarily, then equation \eqref{e-weyltheta} easily simplifies as
\begin{align*}
    M(\ta,\lambda)&=(\ta - M_0(\lambda))^{-1}
=\begin{pmatrix}
\ta_{11}+\frac{u_1^{[0]}(\lambda)}{u_2^{[0]}(\lambda)} & \ta_{12}-\frac{1}{u_2^{[0]}(\lambda)} \\
\ta_{21}-\frac{1}{u_2^{[0]}(\lambda)} & \ta_{22}+\frac{u_2^{[1]}(\lambda)}{u_2^{[0]}(\lambda)}
\end{pmatrix}^{-1}\nonumber
=K(\la)L(\la),\\
\intertext{where}
K(\la) &=\dfrac{u_2^{[0]}(\lambda)}{\left(\ta_{11}u_2^{[0]}(\lambda)+u_1^{[0]}(\lambda)\right)\left(\ta_{22}u_2^{[0]}(\lambda)+u_2^{[1]}(\lambda)\right)-\left(\ta_{12}u_2^{[0]}(\lambda)-1\right)\left(\ta_{21}u_2^{[0]}(\lambda)-1\right)},\\
\intertext{and}
L(\la) &=
\begin{pmatrix}
\ta_{22}u_2^{[0]}(\lambda)+u_2^{[1]}(\lambda) & 1-\ta_{12}u_2^{[0]}(\lambda) \\
1-\ta_{21}u_2^{[0]}(\lambda) & \ta_{11}u_2^{[0]}(\lambda)+u_1^{[0]}(\lambda)
\end{pmatrix}.
\end{align*}
The operator ${\bf L}(\ta)$ has eigenvalues at $\la\ci\ta$ that are poles of $M(\ta)$. Recall that the expressions $u_2^{[0]}(\lambda)$, $u_2^{[1]}(\lambda)$ and $u_1^{[0]}(\la)$ are all finite and well-defined for all values of $\la$ so poles of $M(\ta)$ must occur when the denominator of $K(\la)$ vanishes:
\begin{align}\label{e-thetaeigenvalue}
    \left(\ta_{11}u_2^{[0]}(\lambda)+u_1^{[0]}(\lambda)\right)\left(\ta_{22}u_2^{[0]}(\lambda)+u_2^{[1]}(\lambda)\right)-\left(\ta_{12}u_2^{[0]}(\lambda)-1\right)\left(\ta_{21}u_2^{[0]}(\lambda)-1\right)=0.
\end{align}
Equation \eqref{e-thetaeigenvalue}, which is clearly just $\det(L(\la))$, must therefore be satisfied by $\lambda\ci\ta$. Hence, the rank of $\ta-M_0(\la\ci\ta)$ is at most one. The only way that $\ta-M_0(\la\ci\ta)$ can have nullity $2$ now is if $\la\ci\ta$ is degenerate, as the rows of $L(\la\ci\ta)$ are linearly dependent.
\end{proof}

The characterization of double eigenvalues within $\rho({\bf L}_0)\cup\rho({\bf L}_\infty)$ as degenerate allows for some simple consequences to be established. 

\begin{cor}\label{c-criteria}
Let $K\leq\la\in\rho({\bf L}_0)\cup\rho({\bf L}_\infty)$. Then $\la$ is a double eigenvalue for a single self-adjoint extension ${\bf L}(\ta)$ where $\ta\in\cM$. Likewise, $\la$ is a double eigenvalue for a single self-adjoint extension $\widetilde{\bf L}(\vta)$ where $\vta=-\ta^{-1}\in\cM$. In particular, $\ta$ and $\vta$ must be real, have non-zero off-diagonal entries and be invertible. 
\end{cor}

\begin{proof}
Theorem \ref{t-degen} says that eigenvalues of multiplicity two where $\ta,\vta\in\cM$ are exactly those whose parameters meet the degeneracy conditions of Definition \ref{d-degenerate}. Such matrices $\ta$ and $\vta$ are well-defined as long as $K\leq\la\in\rho({\bf L}_0)\cup\rho({\bf L}_\infty)$ and created by simply plugging in $\la$. Clearly, such a matrix $\ta$ or $\vta$ can correspond to only one value of $\la$. These matrices must be real by Theorem \ref{t-realsimple}, have non-zero off-diagonal entries and are invertible for $\la\in\rho({\bf L}_0)\cup\rho({\bf L}_\infty)$ by the normalization of the fundamental solutions in the Wronskian.
\end{proof}

The Corollary also implies that diagonal matrices cannot produce double eigenvalues, at least for $\la\in\rho({\bf L}_0)\cup\rho({\bf L}_\infty)$. Such matrices correspond to canonical separated boundary conditions, and are well-known to have multiplicity one, see e.g.~\cite[Theorem 10.6.2]{Z} for a modern source where there is no restriction on $\la$.

Later on, in Proposition \ref{p-etazeta}, we analyze domains of the form
\begin{align*}
\{f\in\dom \cD\ti{max} ~:~ t\zeta f^{[0]}(a)=f^{[1]}(a),~ t\eta f^{[0]}(b)=f^{[1]}(b) \},
\end{align*}
for fixed $\eta, \zeta$, and show for which $t$ a value $\lambda \in \rho({\bf L}_0)$ may be an eigenvalue.

Note that Theorem \ref{t-degen} is similar to Theorem \ref{t-olddouble}, which states four requirements for an eigenvalue of a Sturm--Liouville operator with canonical coupled boundary conditions to have multiplicity two, and has no restrictions on $\la$. However, the approach of Theorem \ref{t-degen} via boundary triples still has several advantages. 

\begin{rem}\label{r-improve}
Most importantly, if one is given a canonical coupled boundary condition (which cannot be separated), when using classical methods, it is unclear whether the self-adjoint extension will have any double eigenvalues because explicit values of $u_1^{[0]}(b,\la)$ and $u_2^{[1]}(b,\la)$ are usually hard to calculate. The reduction to checking three such values (this is just due to starting with a Hermitian matrix) instead of four is therefore useful but is not a significant step. There is some guidance on whether the ground state (i.e.~first/lowest eigenvalue) is simple or not via \cite[Theorem 10.6.3]{Z} and whether the extension in real or not via Theorem \ref{t-realsimple} but that is as far as the literature seems to go. The boundary triple approach above is able to offer some basic criteria: if one is given a matrix corresponding to a self-adjoint extension then it must be real, invertible and have non-zero off-diagonal entries to produce an eigenvalue $K\leq \la$ of multiplicity two within $\rho({\bf L}_0)\cup\rho({\bf L}_\infty)$. 
\end{rem}

The use of parameters stemming from boundary triples offers other levels of distinction between different self-adjoint extensions and further clarity. Self-adjoint extensions that are linear relations and not matrices are analyzed in more detail in Subsection \ref{ss-LR}. We will conclude that it is enough to consider the matrices as parameters, and hence the simple yet powerful result of Theorem \ref{t-degen} is even more useful.

\subsection{Eigenvalues of Extensions given by Linear Relations}\label{ss-LR}

Our investigation of the multiplicity of eigenvalues now expands to the case where $\ta$ is a self-adjoint linear relation in $\CC^2$ and not necessarily a matrix. While we will focus on the parametrization of self-adjoint extensions with $\ta$ from Proposition \ref{p-l0}, many of the methods are easily modified to apply the parametrization with $\vta$ from Corollary \ref{c-l1}. Therefore, some of these similar arguments are omitted for brevity. 

Self-adjoint linear relations $\ta$ in $\CC^2$ naturally decompose $\CC^2$ into subspaces where $\ta$ is an operator and where $\ta$ is multi-valued by writing 
\begin{align}\label{e-thetadecomp}
    \cH\ti{op}:=\overline{\dom(\ta)}=\Mul(\ta)^{\perp} \quad\text{and}\quad \cH\ti{mul}:=\Mul(\ta).
\end{align}
Denote the orthogonal projection of $\CC^2$ onto $\cH\ti{op}$ by $P\ti{op}$. It is then possible to write $\ta=\ta\ti{op}\widehat{\oplus}\ta\ti{mul}$ via \cite[Theorem 1.5.1]{BdS}, where
\begin{align*}
    \ta\ti{op}=\left\{\{z,P\ti{op}z'\}~:~\{z,z'\}\in\ta\right\},
\end{align*}
is a self-adjoint operator in $\cH\ti{op}$ and 
\begin{align*}
    \ta\ti{mul}=\left\{\{0,z'\}~:~z'\in\cH\ti{mul}\right\},
\end{align*}
is a self-adjoint purely mutli-valued relation in $\cH\ti{mul}$.
If $\ta$ is a Hermitian matrix, then $\cH\ti{op}=\CC^2$ and $\cH\ti{mul}=\{\emptyset\}$ and such parametrizations were discussed in Section \ref{ss-matrixmult2}. Hence, it now suffices to assume that $\dim\Mul(\ta)\geq 1$. Note that setting $\dim\Mul(\ta)=2$ means that $\ta=\{0,\CC^2\}$. The definition of our boundary triple implies that this linear relation corresponds to the extension ${\bf L}_\infty$, which only has simple eigenvalues by Proposition \ref{p-f}. 

So, without loss of generality, our attention now is focused on the case where $\dim\Mul(\ta)=1 $, where we can write 
\begin{align}\label{e-thetadecomp2}
\ta=\begin{pmatrix}
\ta\ti{op} & 0 \\
0 & \ta\ti{mul}
\end{pmatrix},
\end{align}
which should not be confused with an operator.

The self-adjoint extension ${\bf L}(\ta)$ has an eigenvalue $\la\ci\ta$ of multiplicity two if and only if in the sense of linear relations $\dim\ker(\ta-M_0(\la\ci\ta))=2$ by Proposition \ref{p-dmmult}. Apply a unitary change of basis transformation to $M_0(\la\ci\ta)$ so that it acts on $\cH\ti{op}\oplus\cH\ti{mul}$ via $M\ti{op}$ and $M\ti{mul}$, both of which have domain $\CC^2$ but range $\cH\ti{op}$ and $\cH\ti{mul}$, respectively. It suffices to let $M\ti{op}$ be the linear operator which acts as the first row of $M_0(\la\ci\ta)$ in the new basis and $M\ti{mul}$ act as the second row. In this way, the linear relation resulting from the difference $\ta-M_0(\la\ci\ta)$ can be discussed in the space $\CC^2=\cH\ti{op}{\oplus}\cH\ti{mul}$. Let $z=(z_1,z_2)^T\in\CC^2$. Then, $z\in\ker(\ta-M_0(\la\ci\ta))$ if and only if $z$ belongs to both
\begin{align*}
 \ker(\ta\ti{op}-M\ti{op})&=\ker\left\{\{z,z_1'-M\ti{op}(z)\}~:~\{z_1,z_1'\}\in\ta\ti{op}\right\}  = K_1\cup K_2,\nonumber\\
 \end{align*}
and
\begin{align*}
    \ker(\ta\ti{mul}&-M\ti{mul})=\ker\left\{\{z,z_2'-M\ti{mul}(z)\}~:~\{0,z_2'\}\in\ta\ti{mul}\right\} = K_3\cup K_4,
     \end{align*}
     where
 \begin{align*}
 K_1&=\{z~:~z_1'-M\ti{op}(z)=0\},\\ 
 \qquad
 K_2&=\{z\in\ker M\ti{op}~:~z_1=0\}, \\
    K_3&=\{(z_1,0)^T~:~z_2'-M\ti{mul}[(z_1,0)^T]=0\},\\
    K_4&=\{z\in\ker M\ti{mul}~:~z_2\neq 0\}.
\end{align*}

To clarify, a function $z\in \CC^2$ is an element of $\ker(\ta-M_0(\la\ci\ta))$ if and only if $z\in (K_1\cup K_2)\cap (K_3\cup K_4)$. The next result identifies which combinations of these conditions can lead to double eigenvalues, and is hence a linear relation version of Theorem \ref{t-degen}. While these conditions rely on the transformed $M\ti{op}$ and $M\ti{mul}$, which of course depend on $\ta$ itself, they are still concrete and can be used in the upcoming discussion to show the correspondence between matrix $\ta$ and self-adjoint linear relation $\ta$ (which is not a matrix).

Note that the memberships to $K_1$ through $K_4$ can be modified in the obvious way to apply to $\la\in\rho({\bf L}_\infty)$ by replacing $\ta$ by $\vta$ and $M_0(\la)$ by $M_\infty(\la)$ when decomposing $\CC^2$. 

\begin{theo}\label{t-degenlr}
Let $\ta$ be a self-adjoint linear relation with $\dim\Mul(\ta)=1$ and $\la\ci\ta\in\rho({\bf L}_0)\cup\rho({\bf L}_\infty)$ be an eigenvalue of $L(\ta)$ with multiplicity two. Then there exists some $z,w\in\CC^2$ such that $z\in K_1\cap K_3$, and $w\in K_1\cap K_4$. In particular, $z$ and $w$ are linearly independent.
\end{theo}

\begin{proof}
The previous discussion shows that vectors from the kernel of $\ta-M_0(\la\ci\ta)$ come from linear independent vectors in $\CC^2$ that belong to a combination $K_1$ through $K_4$. However, further scrutiny shows that some combinations do not matter or are not possible. The result is shown for $\la\in\rho(M_0)$, with the analog for $\la\in\rho(M_\infty)$ following similarly. 

Let $\la\in\rho({\bf L}_0)$ and $z\in K_2\cap K_3$. Then clearly $z=(0,0)^T$ and so it does not `influence' $\dim\ker(\ta-M_0(\la\ci\ta))$.

Likewise, let $z\in K_2\cap K_4$. This means that $z_1 = 0$ and $z\in\ker M_0(\la\ci\ta)$. But $M_0(\la)$ has trivial kernel for all $\la\in\CC$ except when $\la\in\sigma({\bf L}_\infty)$ thanks to the relationship $M_\infty(\la)=-M_0(\la)^{-1}$ when $\la\in\rho({\bf L}_0)$. Hence, $\la\ci\ta\in\sigma({\bf L}_\infty)$ and by Proposition \ref{p-f} $(ii)$, $u_1^{[1]}(b,\la\ci\ta)=0$ which in turn implies that $u_1^{[0]}(b,\la\ci\ta)\neq0$ and $u_2^{[1]}(b,\la\ci\ta)\neq0$ due to the normalization of the fundamental system of solutions. Suppressing the dependence of solutions on $b$, this explicitly means
\begin{align*}
    \begin{pmatrix}
    0 \\
    0
    \end{pmatrix}
    &=\frac{1}{u_2^{[0]}(\la\ci\ta)}
    \begin{pmatrix}
    -u_1^{[0]}(\la\ci\ta) & 1 \\
    1 & -u_2^{[1]}(\la\ci\ta)
    \end{pmatrix}
    \begin{pmatrix}
    0 \\
    z_2
    \end{pmatrix} \\
    &=\frac{1}{u_2^{[0]}(\la\ci\ta)}
    \begin{pmatrix}
    z_2 \\
    -z_2 u_2^{[1]}(\la\ci\ta)
    \end{pmatrix}.
\end{align*}
Clearly, this is only possible if $z_2=0$. We conclude that $z\in K_2\cap K_4$ does not `influence' $\dim\ker(\ta-M_0(\la\ci\ta))$.

Now, $\la\ci\ta\in\rho({\bf L}_0)\cup\rho({\bf L}_\infty)$ and has multiplicity two so $\dim\ker(\ta-M_0(\la\ci\ta))=2$. There are only two combinations of conditions remaining and the result follows. 
\end{proof}

In Floquet-Bloch theory (see e.g.~\cite{BK13, Kuch93}), unions of eigenvalues for different boundary conditions are of interest to determine the spectrum of graph Laplacians and Laplacians on infinite periodic lattices. Roughly speaking, the union (over varying boundary conditions) of eigenvalues of the operator over a fundamental domain equals the (purely absolutely continuous) spectrum of the Laplacian on the infinite periodic graph. Motivated by such applications, we consider the unions of eigenvalues for the Sturm--Liouville setting:
\begin{align*}
    \Sigma^\cM:=\bigcup_{\ta\in\cM}\sigma({\bf L}(\ta)),
    \quad
    \Sigma^\mathcal{E}:=\sigma({\bf L}_0)\cap\sigma({\bf L}_\infty)
    \quad\text{ and }\quad 
    \Sigma^\cR:=\left(\bigcup_{\ta\in\cR}\sigma({\bf L}(\ta))\right)\backslash\left(\Sigma^\cM\cup\Sigma^\mathcal{E}\right).
\end{align*}

The notation should not be confused with the singular or residual parts of the spectrum of extensions. The analysis conducted in the proof of Theorem \ref{t-degen} reveals that those eigenvalues in $\Sigma^\cR$ are also eigenvalues of a corresponding Hermitian $2\times 2$ matrix. The following result follows directly from Corollary \ref{c-criteria}.

\begin{cor}\label{c-nonewrelations}
Any eigenvalue $\la\ci\ta\in\rho({\bf L}_0)\cup\rho({\bf L}_\infty)$ of ${\bf L}(\ta)$ for $\ta$ a self-adjoint linear relation that is not a matrix belongs to $\Sigma^\cM$. In particular, this means that $\Sigma^\cR=\emptyset$.
\end{cor}

\begin{proof}
Corollary \ref{c-criteria} says that any $K\leq\la\in\rho({\bf L}_0)\cup\rho({\bf L}_\infty)$ can be realized not only as an eigenvalue of some ${\bf L}(\ta)$ with $\ta$ an invertible matrix, but as an eigenvalue of multiplicity two. Hence, any eigenvalue stemming from a ${\bf L}(\ta)$ where $\ta$ is a self-adjoint linear relation that is not a matrix falls into $\Sigma^\cM$ by definition. 
\end{proof}

However, it is possible to say more about these eigenvalues that come from choices of parameters that are self-adjoint linear relations and not matrices by exploiting the decomposition of $\CC^2$ into one-dimensional subspaces corresponding to $\ta\ti{op}$ and $\ta\ti{mul}$, as above, see equation \eqref{e-thetadecomp2}. We can concretely identify matrices which will create such eigenvalues. 

Again, we focus our attention on $\ta$ such that $\dim\Mul\ta=1$ and let $\la\ci\ta$ be an eigenvalue of ${\bf L}(\ta)$. Decompose $M_0(\la\ci\ta)$ into $M\ti{op}$ and $M\ti{mul}$ as above and let $z=(z_1,z_2)^T\in\CC^2$. Proposition \ref{p-dmmult} says that $\la\in\sigma({\bf L}\ci\ta)$ if and only if $0\in\sigma(\ta-M_0(\la\ci\ta))$. The proof of Theorem \ref{t-degenlr} showed that the corresponding $z$ belongs to either $K_1\cap K_3$ or $w\in K_1\cap K_4$.

For a moment assume that $z\in K_1\cap K_4$, so that $\ta\ti{op}z_1=M\ti{op}[(z_1,z_2)^T]$, $M\ti{mul}[(z_1,z_2)^T]=0$ and $z_2\neq 0$. Now, consider $M_0\left(\la\ci{\ta}\right)$ in the same basis but with 
\begin{align*}
\widetilde{\ta}=
\begin{pmatrix}
\ta\ti{op} & 0 \\
0 & 0
\end{pmatrix}.
\end{align*}
We clearly have $0\in\sigma\left(\widetilde{\ta}-M_0\left(\la\ci{\ta}\right)\right)$. Therefore, ${\bf L}(\widetilde{\ta})$ has $\la\ci{\ta}$ as an eigenvalue and $\la\ci{\ta}\in\Sigma^\cM$.

Alternatively, assume that $\widetilde{w}\in K_1\cap K_3$. Then consider $M_0\left(\la\ci{\ta}\right)$ in the same basis but set
\begin{align*}
\widetilde{\ta}=
\begin{pmatrix}
\ta\ti{op} & 0 \\
0 & M\ti{mul}[(\widetilde{w}_1,0)^T]
\end{pmatrix}.
\end{align*}
Such a $\widetilde{\ta}\in\cM$ will have $\la\ci\ta$ as an eigenvalue and we again can see that $\la\ci\ta\in\Sigma^\cM$.

Essentially, this means the behavior of eigenvalues of ${\bf L}(\ta)$ arising from $\ta$ with $\dim\Mul\ta=1$ can be mimicked by $2\times 2$ matrices for the purposes of attaining an eigenvalue in a more convenient form. In this section, these alternative representations of linear relations as matrices are valid only for eigenvalues of multiplicity one, but eigenvalues of multiplicity two can simply be recreated with matrices according to Definition \ref{d-degenerate} and Theorem \ref{t-degen}.

\section{Aronszajn--Donoghue and Sturm--Liouville operators}\label{s-AD}

The semi-bounded Sturm--Liouville expression $\ell$ had its self-adjoint extensions naturally parameterized by boundary triples in Section \ref{s-twolc}. However, there are additional tools for the spectral analysis of such operators in perturbation theory and so we reformulate the self-adjoint extensions of $\ell$ so that these tools can be applied.

We start with the self-adjoint operator ${\bf L}_0$ (defined in equation \eqref{e-l0}), and add a rank-two perturbation which is parametrized by $\vartheta$ and changes the correct subspace of $\cH$ by coordinate mappings ${\bf B}$ and ${\bf B^*}$. The perturbation is well defined due to \cite[Theorem 3.6]{BFL}, but the justification relies on tools including a boundary pair and singular space $\cH_{-1}\supseteq\cH$ \cite{AK} (note the subscripts on the below inner products correspond to $\cH_1\subseteq\cH$). These details do not have an impact on the spectral analysis and so we refer readers to \cite{BFL} for more information. 

We define 
\begin{equation*}
\begin{aligned}
\langle f,\widetilde{\delta}_a\rangle_{1}&=\lim_{x\to a^+}\dfrac{f(x)}{v_a(x,\la_0)}, \\
\langle f,\widetilde{\delta}_b\rangle_{1}&=\lim_{x\to b^-}\dfrac{f(x)}{v_b(x,\la_0)},
\end{aligned}
\end{equation*}
so that the coordinate mapping ${\bf B}:\CC^2\to\Ran({\bf B})\subset\cH_{-1}({\bf L}_0)$ acts via multiplication by the row vector
$    
\begin{pmatrix}
    \widetilde{\delta}_a, \widetilde{\delta}_b
\end{pmatrix}.
$
A simple calculation yields that on $\cH_{-1}({\bf L}_0)$ the operator ${\bf B}^*$ is given by
\begin{align*}
    {\bf B}^*f=
    \begin{pmatrix}
    \langle f,\widetilde{\delta}_a\rangle_{1} \\
    \langle f,\widetilde{\delta}_b\rangle_{1}
    \end{pmatrix}.
\end{align*}
The parametrization is accomplished via self-adjoint linear relations in $\CC^2$, denoted $\vta$ and defined as $\vta=-\ta^{-1}$, for $\ta$ as in Section \ref{s-twolc} and Remark \ref{r-alternative}.

The perturbation can then be defined via \cite[Theorem 3.6]{BFL} so that
\begin{align}\label{e-pertsetup}
\wt{\bf L}\civta={\bf L}_0 + {\bf B}\vta{\bf B^*},
\end{align}
and every self-adjoint extension of the minimal operator ${\bf L}\ti{min}$ can be written as $\wt{\bf L}\civta$ for some $\vta\in \cR$. In this section, however, we consider only $\vartheta\in \cM.$

To avoid possible confusion, we rephrase some of Remark \ref{r-alternative}: The main reason for the use of $\vta$ here, as opposed to $\ta$, is because boundary triples and perturbation theory naturally have opposite starting points, or self-adjoint extensions, that emerge. The choice of $\vta$ above allows ${\bf L}_0$ to be the natural self-adjoint extension of interest in both cases, as opposed to the notation in \cite{BFL}.

Let us begin our spectral analysis by briefly discussing ramifications of two Aronszajn--Donoghue type results in our setting of finite rank perturbations obtained by Sturm--Liouville operators. To state those results, we consider rank $d$ perturbations of a self-adjoint operator $A$ on a separable Hilbert space $\cH.$ Focusing on the interesting part, such perturbations can be represented by  $A\ci\vta = A+\bB \vta\bB^*,$ $\bB:\C^d\to \cH,$ with $\Ran \bB$ being cyclic and where $\vta$ is a $d\times d$ Hermitian matrix. By $\bmu_\vta$ and $\mu_\vta,$ respectively, we denote the matrix and scalar spectral measures of $A_\vta$ with respect to $\bB.$ This just means
\begin{align}\label{e-mmmm}
\wt M (\vta,\la) = \bB^*(A\ci\vta-\la)^{-1}\bB
=
\int\ci\R \frac{d\bmu_\vta(t)}{t-\la}
\quad\text{and}\quad
\mu\ci\vta = \Tr \bmu\ci\vta.
\end{align}

\begin{rem}
The definition of $\bB$ agrees with the choices of boundary maps. Therefore, the setup for the unperturbed operator ${\bf L}_0$ agrees with that in Proposition \ref{p-l0}. Let us verify the compatibility of the $m$-functions for the other self-adjoint extensions used in the boundary triples set up in the previous sections and in perturbation theory. This is done by confirming that the $m$-functions in \eqref{e-Mfunction} and \eqref{e-mmmm}, both of which are denoted by $\wt M (\vta,\la),$ do indeed agree:

A central object of finite rank perturbation theory is the Aronszajn--Krein type formula $\wt M (\vta,\la) = M_0(\la)(\vartheta M_0(\la)+ I)^{-1},$ see e.g.~\cite{S} for the rank one setting and \cite{LT_JST} for higher rank perturbations. On the other hand, we recall the relation  $\wt M (\vta,\la) = (\vartheta - M_\infty(\la))^{-1}$  of $m$-functions given in \eqref{e-Mfunction}. This apparent difference is reconciled by simply observing that $M_0 = -M_\infty^{-1}$ implies 
$$
M_0(\vartheta M_0+ I)^{-1} = -M_0(-\vartheta M_0- I)^{-1} = M_\infty^{-1}(\vartheta M_\infty^{-1}- I)^{-1} = (\vartheta - M_\infty)^{-1}. 
$$
Here we suppressed the explicit dependence of the $m$-functions on $\lambda$ for the convenience of readability.
\end{rem}

We are now ready to state and discuss aforementioned Aronszajn--Donoghue type results.

\begin{theo}[{\cite[Theorem 1.1]{LTV_IMRN}}]\label{t-scalarAD}
Let $\nu$ be a singular measure. The scalar spectral measures $\mu\ci\vta$ of finite rank $d$ perturbations $A\ci\vta$ with cyclic $\Ran \bB$ are mutually singular with $\nu$ for all parameters $\vta$, except (possibly) a set of Hausdorff dimension $d^2-1$.
\end{theo}

In other words, by Theorem \ref{t-scalarAD}, the dimension of the scalar (by taking the traces of the matrix-valued) spectral measures $\mu_\vta$ may be $\nu\not\perp\mu_\vta$ for at most $d^2-1$ dimensions. Setting $\nu = \mu\ci{\vta_0}$ for $\vta_0\neq\vta$, we obtain an Aronszajn--Donoghue type result.

In our application to Sturm--Liouville operators, we have $d=2$, making the parameter space $d^2=4$ dimensional, with a $d^2-1 = 3$ dimensional boundary associated with those boundary conditions that correspond to perturbation parameters from $\mathcal{R}\setminus\mathcal{M}$. And so, even if Theorem \ref{t-scalarAD} could be extended from $\mathcal{M}$ to $\mathcal{R}$, we would not be able to gain any useful information regarding the mutual singularity of the scalar spectral measures for parameters in $\mathcal{R}\setminus\mathcal{M}$.

For lines in parameter space we know more:
\begin{theo}[{\cite[Theorem 7.3]{LT_JST}}]\label{t-linesAD}
Let $\nu$ be a singular Radon measure, $\wt\vta\in\cM$ and let $\vta\in \mathcal{M}$ be positive definite. Then scalar spectral measures $\mu_{\wt\vta+t\vta}$ of $A\ci{\wt\vta+t\vta}$ are mutually singular with $\nu$ for all except (possibly) countably many $t\in \R$.
\end{theo}

\begin{rem}\label{r-after}
Notice that simple examples can show that the condition of $\vta$ being positive definite cannot be relaxed to $\vta$ being non-negative. Indeed, perturbing any diagonal $2\times 2$ matrix by $\vta = \begin{pmatrix}0&0\\0&\eta\end{pmatrix}$ will leave one of the eigenvalues unchanged. To avoid possible confusion, we mention that this example does not contradict the classical Aronszajn--Donoghue Theorem, see e.g.~\cite{S}, since this would be a rank one perturbation by a non-cyclic vector.
\end{rem}

\begin{proof}[Proof of Proposition \ref{p-varnothing}]
For some $\ta_0\in \cR$ let $A_0 ={\bf L}(\ta_0)$.  We apply Theorem \ref{t-linesAD} with $A_0,$ $\tilde \vta = 0$ and some (just any) positive definite $\vta\in \cM$ as well as $\nu $ being equal to the scalar-valued spectral measure of ${\bf L}(\ta_0)$ with respect to $\bB$. Then, by the conclusion of Theorem \ref{t-linesAD}, the scalar spectral measure of operator $A_{t\vta}$ (with respect to $\bB$) is mutually singular with respect to $\nu$ for uncountably many $t\in \R.$ We simply let $\ta_1$ be such that ${\bf L}(\ta_1) = A_{t\vta}$ for one of those uncountably many $t\in\R$. Since boundary triples ${\bf L}(\ta)$ exhaust all self-adjoint extensions, it is always possible to find such a $\ta_1\in \cR.$ And because in our setting of Sturm--Liouville operators the spectrum consists of eigenvalues, we have $\sigma({\bf L}(\ta_0))\cap \sigma({\bf L}(\ta_1))= \varnothing.$
\end{proof}

\begin{rem}
Operator $A_{t\vta}$ is a perturbation of operator $A$ by $t\vta\in \cM,$ so that Theorem \ref{t-linesAD} can indeed be applied although we might not be able to attain ${\bf L}(\ta_0)$ via a  perturbation $\wt{\bf L}(\vta)$ as defined in \eqref{e-pertsetup} with $\vta\in \cR\setminus \cM$.
\end{rem}

Motivated by Theorem \ref{t-linesAD}, we explore Sturm--Liouville operators with boundary conditions that correspond to perturbations with parameters of the form $\wt\vta+t\vta$ with $\wt\vta, \vta\in \cM.$ Our results greatly improve the conclusions of Theorem \ref{t-linesAD} in the case of Sturm--Liouville operators. One major simplification comes from the fact that in this setting the spectrum is pure point.

\begin{theo}\label{t-tatta}
Consider the Sturm--Liouville operator $\wt{\bf L}(\wt\vta+t\vta)$ with fixed $\wt\vta, \vta\in \cM,$ $\det\vta \neq 0.$ Then every $\lambda\in \rho({\bf L}_\infty)$ is an eigenvalue for at most two values of $t\in \R.$
\end{theo}

We obtain an immediate corollary in the spirit of Theorem \ref{t-linesAD}.

\begin{cor}
Consider the Sturm--Liouville operator $\wt{\bf L}(\wt\vta+t\vta)$ with fixed $\wt\vta, \vta\in \cM,$ $\det\vta \neq 0.$ Let $\nu$ be a singular measure with $\nu\perp\mu_\infty$. Then scalar spectral measures $\mu\ci{\wt\vta+t\vta}$ of $\wt{\bf L}(\wt\vta+t\vta)$ are mutually singular with $\nu$ for all except (possibly) two $t\in \R$.
\end{cor}

For convenience, we suppress the independent variables in functions $u_i^{[j]}(b,\lambda),$ and write $u_i^{[j]}$ instead.

\begin{proof}[Proof of Theorem \ref{t-tatta}]
By equation \eqref{e-wtm} and the formula for $M_\infty$ in Corollary \ref{c-l1}, the $M$-function corresponding to $\wt{\bf L}(\wt\vta+t\vta)$ is formally given by
\begin{align*}
    \wt M(\wt\vta+t\vta,\la) = (\wt\vta+t\vta-M_\infty(\lambda))^{-1}
    =
    \frac{u_1^{[1]}}{\det X}X^{-1},
\end{align*}
where
\begin{align*}
    X = 
\begin{pmatrix}
\wt\vta_{11} u_1^{[1]} + t\vta_{11} u_1^{[1]}+u_2^{[1]}
&
\wt\vta_{12} u_1^{[1]} + t\vta_{12} u_1^{[1]}-1
\\
\wt\vta_{21} u_1^{[1]} + t\vta_{21} u_1^{[1]}-1
&
\wt\vta_{22} u_1^{[1]} + t\vta_{22} u_1^{[1]}+u_1^{[0]}
\end{pmatrix}.
\end{align*}
As before, the eigenvalues of $\wt{\bf L}(\wt\vta+t\vta)$ occur at the roots of $\det X.$ And, clearly, $\det X$ is a polynomial of degree two in $t$ with leading coefficient $u_1^{[1]}\det\vta$. By Proposition \ref{p-f}, we have $u_1^{[1]}\neq 0$ and by our assumption we have $\det\vta\neq 0.$ So $\det X = 0$ for zero, one or two values of $t.$
\end{proof}

A somewhat tedious calculation yields two immediate corollaries to this proof.

\begin{cor}\label{r-tedious}
When $\det\vta=0$ in Theorem \ref{t-tatta} while $$c:=\wt\vta_{11}\vta_{22}u_1^{[1]}+\vta_{11}\wt\vta_{22}u_1^{[1]}-\vta_{11}u_1^{[0]}-\vta_{22}u_2^{[1]}-\wt\vta_{12}\vta_{12}u_1^{[1]}-\wt\vta_{21}\vta_{21}u_1^{[1]}+\vta_{12}+\vta_{21}\neq 0,$$ then every $\lambda \in \rho({\bf L}_\infty)$ is an eigenvalue for exactly one value of $t\in \R.$

This value is $t = \frac{d}{c},$ where 
\begin{align*}
d: = \det(\wt\vta) u_1^{[1]}-\wt\vta_{11}u_1^{[0]}-\wt\vta_{22}u_2^{[1]}+ \wt\vta_{12}+ \wt\vta_{21}+ u_2^{[0]}.
\end{align*}
\end{cor}

\begin{cor}
When both $\det\vta = 0$ in Theorem \ref{t-tatta} and $c=0$ (from Corollary \ref{r-tedious}), then $\lambda \in \rho({\bf L}_\infty)$ is either an eigenvalue for all $t$ or for none (respectively, if $d=0$ or $d\neq 0$).
\end{cor}

\begin{prop}\label{p-tatta}
Fix $\wt\vta, \vta\in \cM$ with non-zero $\vta_{11}, \vta_{12}, \vta_{22}$. Then an eigenvalue $\lambda \in \rho({\bf L}_\infty)$ of $\wt{\bf L}(\wt\vta+t\vta)$ has multiplicity 2 only if 
\begin{align}\label{a-tatta}
    t=\frac{-\wt\vta_{11} u_1^{[1]} +u_2^{[1]}}{\vta_{11} u_1^{[1]}}
=\frac{-\wt\vta_{12} u_1^{[1]} +1}{\vta_{12} u_1^{[1]}}
=\frac{-\wt\vta_{21} u_1^{[1]}+1}{\vta_{21} u_1^{[1]}}
=\frac{-\wt\vta_{22} u_1^{[1]}+u_1^{[0]}}{\vta_{22} u_1^{[1]}}.
\end{align}
In particular, operator $\wt{\bf L}(\wt\vta+t\vta)$ has $\lambda \in \rho({\bf L}_\infty)$ with multiplicity two at most for one value of $t.$
\end{prop}

\begin{proof}
By Proposition \ref{p-dmmult}, we have $\dim\ker\left(\wt{\bf L}(\wt\vta+t\vta)-\la\right)=\dim\ker\left(\wt\vta+t\vta-M_\infty(\la)\right)$. Therefore, the multiplicity of an eigenvalue is two only if all entries of $\wt\vta+t\vta-M_\infty(\la)$ satisfy
\begin{align*}
    \wt\vta_{11} u_1^{[1]} + t\vta_{11} u_1^{[1]}-u_2^{[1]}
=
\wt\vta_{12} u_1^{[1]} + t\vta_{12} u_1^{[1]}-1
\\
=\wt\vta_{21} u_1^{[1]} + t\vta_{21} u_1^{[1]}-1
=
\wt\vta_{22} u_1^{[1]} + t\vta_{22} u_1^{[1]}-u_1^{[0]}=0.
\end{align*}
For $\lambda \in \rho({\bf L}_\infty),$ we have $u_1^{[1]}\neq 0$ by Proposition \ref{p-f}.
Solving each expression for $t$ while keeping in mind that $\vta_{21} = \overline{\vta_{12}}\neq 0$ proves the result.
\end{proof}

In fact, we notice that this proof is an extension of the proof of Corollary \ref{c-criteria}. The next result combines the findings of the proof of Theorem \ref{t-tatta} and Proposition \ref{p-tatta}.
\begin{theo}\label{t-tattaIFF}
Fix $\wt\vta, \vta\in \cM$ with non-zero $\vta_{11}, \vta_{12}, \vta_{22}$ and $\det \vta\neq 0.$ Then an eigenvalue $\lambda \in \rho({\bf L}_\infty)$ of $\wt{\bf L}(\wt\vta+t\vta)$ has multiplicity 2 if and only if the conditions in \eqref{a-tatta} are satisfied and 
\begin{align}\label{a-leftover}
\left(u_1^{[1]}-1\right)\left(\wt\vta_{12}+\wt\vta_{21}+u_2^{[1]}u_1^{[0]}- \wt\vta_{22}u_2^{[1]}-\wt\vta_{11}u_1^{[0]}\right)
+ u_1^{[1]}(2\det \vta +1)
=0.
\end{align}
\end{theo}

\begin{proof}
This theorem follows, once again after a tedious calculation, by substituting the appropriate expressions for $t$ from conditions \eqref{a-tatta} into the equation $\det X = 0$. As this calculation goes on, it turns out that many terms cancel, and the remaining terms are listed in \eqref{a-leftover}.
\end{proof}

For a moment, we consider even simpler perturbations in order to relate to Corollary \ref{c-criteria}. Here we obtain explicit values of $t$.

\begin{cor}\label{p-etazeta}
The Sturm--Lioville operator $\wt{\bf L}(t\vta)$, where $\vta = \begin{pmatrix}\zeta &0\\0&\eta\end{pmatrix}$ with $\eta,\zeta\neq 0$ has an eigenvalue at $\lambda\in \rho({\bf L}_\infty)$, if
$$
t = \frac{\eta u_2^{[1]}+\zeta u_1^{[0]}\pm\sqrt{4\eta\zeta + (\eta u_2^{[1]}-\zeta u_1^{[0]})^2}}{2\zeta\eta u_1^{[1]}}.
$$
\end{cor}

\begin{proof}
This explicit expression for $t$ is obtained by solving $\det X=0$ (from the proof of Theorem \ref{t-tatta}) for $t$.
\end{proof}

\begin{prop}\label{p-simpleeverywhere}
Every real $\la\geq K$ is attained as a simple eigenvalue of a Sturm--Liouville operator with canonical separated boundary conditions.

Moreover, $K\leq\lambda\in \rho({\bf L}_\infty)$ is an eigenvalue for $\widetilde{\bf L}(\vta)$ where $\vartheta \in \cM\setminus\{0\}.$ 
\end{prop}

\begin{proof}
Note that ${\bf L}_\infty$ has simple eigenvalues and corresponds to Sturm--Liouville operator with separated boundary conditions. So the statement is true for $\lambda\in \sigma({\bf L}_\infty)$. Further, recall that every eigenvalue of a Sturm--Liouville operator with canonical separated boundary conditions is simple by Corollary \ref{c-criteria}.

For real $K\leq \lambda\in \rho({\bf L}_\infty)$ it therefore suffices to show that there are values of $\eta, \zeta$ that produce a real $t$ in Corollary \ref{p-etazeta}.
At such a fixed $\la,$ the functions $u_2^{[1]}$ and $u_1^{[0]}$ take on finite real values. In this case, in accordance with Corollary \ref{p-etazeta}, the proposition follows if we can show that the discriminant $D = 4\eta\zeta + (\eta u_2^{[1]}-\zeta u_1^{[0]})^2$ is non-negative for some choice of $\eta,\zeta\neq 0$. Since $\eta^2 (u_2^{[1]})^2+\zeta^2 (u_1^{[0]})^2\ge 0,$ we obtain that
$
D\ge 2\eta\zeta(2-\eta\zeta u_2^{[1]}u_1^{[0]}).
$
Now we take any $\eta$ and $\zeta$ that are of the same sign and with magnitude so that $2>\eta\zeta u_2^{[1]}u_1^{[0]}.$
\end{proof}

On the flip side, note that if $D<0,$ then there are no real values of $t$. This can be used to imply that certain values of $\la$ may not be eigenvalues for `many' Sturm-Liouville operators with separated boundary conditions.

Proposition \ref{p-simpleeverywhere} can also be shown by removing the degeneracy condition from Definition \ref{d-degenerate} on the entry $\vta_{22}$ so that $\dim\ker(\vta-\widetilde{M}(\la))=1$. The above presentation demonstrates how perturbation theory adds a different flavor to these results.

\section{Spectral representations of the perturbed operators $\wt{\bf L}(\vta)$}\label{s-SpecRep}

We reproduce some results that were obtained in the case of the Jacobi operator in \cite{BFL} for the setting of semi-bounded Sturm--Liouville operators with two limit-circle endpoints. Although the location of point masses of the general Sturm--Liouville operators ${\bf L}_0$ and $\wt{\bf L}(\vta)$ are only known in examples, we can retrieve some other rather concrete information. To do so, we briefly recall a well-known result by Gesztesy--Tsekanovski which will allow us to extract spectral information from the Weyl $M$-function. 
\begin{theo}{\cite[Theorem 5.5(i)]{GT}}\label{t-GesTs}
For a matrix-valued Herglotz function with representation
\begin{align}\label{e-Mfunction}
    M(z) = C + Dz + \int_\mathbb{R} ((w-z)^{-1} - w(1+w^2)^{-1})d \bmu(w),
\quad z\in \mathbb{C}_+,
\end{align}
we have
\begin{align*}
\bmu\{\lambda\} = -i\lim_{\eps\searrow 0} \eps M(\lambda +i\eps).
\end{align*}
\end{theo}

\begin{theo}\label{t:mu}
The matrix-valued weights of the point masses $\la_n$ of ${\bf L}_0$ are given by
\begin{align}\label{e-bmulambdan}
    \bmu\{\lambda_n\} 
    &=
    -i\lim_{\eps\searrow 0}\eps \dfrac{1}{u_2^{[0]}(b,\lambda_n+i\eps)}
    \left( \begin{array}{cc}
-u_1^{[0]}(b,\lambda_n+i\eps) & 1 \\
1 & -u_2^{[1]}(b,\lambda_n+i\eps)
\end{array} \right).
\end{align}
In particular, we verify the well-known fact that the multiplicity of each eigenvalue is one.
\end{theo}

\begin{proof}
From Proposition \ref{p-l0}, we recall
\begin{align*}
    M_0(\lambda)=\dfrac{1}{u_2^{[0]}(b,\lambda)}
    \left( \begin{array}{cc}
-u_1^{[0]}(b,\lambda) & 1 \\
1 & -u_2^{[1]}(b,\lambda)
\end{array} \right),
\end{align*}
for $\lambda\in\rho({\bf L}_0)$. And in accordance with Theorem \ref{t-GesTs} we set out to evaluate the limit of $\eps M_0(\lambda_n +i\eps)$ as $\eps\searrow 0$. Here it makes sense to evaluate $M_0$ at $\la_n+i\eps\in \rho({\bf L}_0)$. We note that this limit exists finitely by the choice of the functions $u_1^{[0]}(b,\lambda)$ and $u_2^{[1]}(b,\lambda)$, and, in particular, since we can `at most' exhibit eigenvalues. By `at most' we mean that the limit 
\begin{align*}
\lim_{\eps\searrow 0} \dfrac{\eps}{u_2^{[0]}(b,\lambda+i\eps)},
\end{align*}
can either be zero (when $\la$ is not an eigenvalue of ${\bf L}_0$) or some non-zero constant. We obtain \eqref{e-bmulambdan}.

We have ${\bf L}_0 = \wt{\bf L}(0)$, and so by Proposition \ref{p-dmmult} having simple eigenvalues is equivalent to $\dim\Ker(M_\infty(\la_n))\neq 0$ for each $\la_n$. This is clear from the fact that by Corollary \ref{c-l1} the matrix function $M_\infty(\la)$ has non-zero off-diagonal entries for all $\la$.
\end{proof}

Next we consider the unperturbed operator ${\bf L
}_0$ in its spectral representation, which is given by the operator $\boldsymbol{M}$ on $L^2(\bmu)$ (the operator that multiplies by the independent variable $\boldsymbol{M}:L^2(\bmu)\to L^2(\bmu):f(x)\mapsto xf(x)$).
\begin{cor}\label{c:l2mu}
The eigenvector of ${\bf L}_0$ corresponding to eigenvalue $\la_n$ takes the following form in the spectral representation
\begin{align*}
f_n(x)=\chi\ci{\{\la_n\}}(x)\begin{pmatrix} 1\\-\lim_{\eps\searrow 0}u_2^{[1]}(b,\lambda_n+i\eps)\end{pmatrix}\in L^2(\bmu),
\end{align*}
where $\chi\ci{\{\lambda_n\}}$ denotes the characteristic function at $x=\la_n$. We obtain an explicit formulation of the space $L^2(\bmu)=\clos \spa \left\{f_n:n\in \N_0\right\}.$
\end{cor}

\begin{proof}
Since the multiplicity of each eigenvalue equals one, the columns of \eqref{e-bmulambdan} are linearly dependent. And neither of the columns being trivial implies that we can just pick one of them. The magnitude of the eigenvector does not matter, so we can simply drop the multiplicative constant from \eqref{e-bmulambdan}. 
\end{proof} 

For operators $\wt{\bf L}(\vta)$ we obtain the following result using analogous arguments and \eqref{e-wtm}.

\begin{cor}\label{c:EVE2}
Assume that $\vta\in\cM$ corresponds to a non-degenerate operator $\wt{\bf L}(\vta).$ The eigenvector $f^\vta_n\in L^2(\bmu_\vta)$ of $\wt{\bf L}(\vta)$ corresponding to an eigenvalue $\la^{\vta}_n,$ $n\in \N_0,$ takes the following form in the spectral representation $$f^{\vta}_n(x) =\chi\ci{\{\lambda^{\vta}_n\}}(x)\begin{pmatrix} 1-\vta_{12}\\\vta_{11}-\lim_{\eps\searrow 0}u_2^{[1]}(b,\lambda_n+i\eps)\end{pmatrix}.$$
The spectral representation lives on the space $L^2(\bmu^\vta)=\clos \spa \left\{f_n^\vta:n\in \N_0\right\}.$
\end{cor}

\begin{rem}
\label{r-degenerate1}
For a degenerate eigenvalue $\lambda_n^\vta$ the corresponding eigenspace in the spectral representation is spanned by $\chi\ci{\{\lambda^{\vta}_n\}}(x)\begin{pmatrix} 1\\0\end{pmatrix}\in L^2(\bmu^\vta)$ and $ \chi\ci{\{\lambda^{\vta}_n\}}(x)\begin{pmatrix} 0\\1\end{pmatrix}\in L^2(\bmu^\vta).$
\end{rem}


\end{document}